\documentclass[a4paper,oneside,openany,11pt]{article}
\usepackage[nointlimits]{amsmath}
\usepackage{amsfonts}
\usepackage{amssymb}
\usepackage{amsmath}
\usepackage{amsthm}
\usepackage{enumerate}
\usepackage{latexsym}
\usepackage{graphicx}
\usepackage{layout}
\usepackage[english]{babel}
\def \rr {\mathbb{R}}

\def \nn {\mathbb{N}}

\def \beqn {\begin{equation}}
\def \eeqn {\end{equation}}
\def \bi {\begin{itemize}}
\def \ei {\end{itemize}}
\def \ben {\begin{enumerate}}
\def \een {\end{enumerate}}
\def \beq {\begin{eqnarray*}}
\def \eeq {\end{eqnarray*}}
\def \beqn {\begin{eqnarray}}
\def \eeqn {\end{eqnarray}}

\def\coz{z}

\def\mz2a{\pabs \coz\pabd^{2\al}}

\newcommand{\ba}{\begin{array}}
\newcommand{\ea}{\end{array}}
\newcommand{\beqa}{\begin{eqnarray}}
\newcommand{\eeqa}{\end{eqnarray}}

\newtheorem{theo}{Theorem}[section]
\newtheorem{lemma}[theo]{Lemma}
\newtheorem{prop}[theo]{Proposition}
\newtheorem{coro}[theo]{Corollary}

\newtheorem{remark}[theo]{Remark}

\begin{document}

\title{On non-topological solutions for planar Liouville Systems of Toda-type}

\author{Arkady Poliakovsky \thanks{Department of Mathematics, Ben-Gurion University of the Negev, P.O.B. 653, Beer-Sheva 84105, Israel. Email: poliakov@math.bgu.ac.il} \quad Gabriella Tarantello \thanks{Dipartimento di Matematica.
Universit\`a degli Studi di Roma "Tor Vergata", Via della Ricerca Scientifica, 00133 Rome, Italy. Email: tarantel@mat.uniroma2.it} 
\thanks{Supported by PRIN09 project: \textit{Nonlinear elliptic problems in the study of vortices and related topics},  PRIN12 project: \textit{Variational and Perturbative Aspects of Nonlinear Differential Problems} and  FIRB project: \textit{Analysis and Beyond.}}}


\maketitle

\begin{abstract}
Motivated by the study of non-abelian Chern Simons vortices of non-topological type in Gauge Field Theory, see e.g. \cite{gud1, gud2},\cite{dunne}, we analyse the solvability of the following (normalised) Liouville-type system in presence of singular sources:

\begin{equation*} 
(1)_\tau \begin{cases}
-\Delta u_1 = e^{u_1} -  \tau  e^{u_2} - 4N\pi \, \delta_0,\\
-\Delta u_2 = e^{u_2} -  \tau  e^{u_1}, \\
\beta_1 = \frac{1}{2\pi} \int_{\mathbb{R}^2} e^{u_1} \; \text{ and } \; \beta_2 = \frac{1}{2\pi} \int_{\mathbb{R}^2} e^{u_2}, 
\end{cases}
\end{equation*}

with $\tau>0$ and $N>0$.\\
We identify \underline{ necessary and sufficient} conditions on the parameter $\tau$ and the "flux" pair: $(\beta_1, \beta_2),$  which ensure the \underline{radial} solvability of $(1)_\tau.$\\
Since for $\tau=\frac{1}{2},$ problem $(1)_\tau$ reduces to the (integrable) 2 X 2 Toda system, in particular we recover the existence result of  \cite{lwy} and \cite{jw}, concerning this case.\\
Our method relies on a blow-up analysis  for solutions of $(1)_\tau$, which (even in the radial setting) takes new turns compared to the single equation case.\\
We mention that our approach permits to handle also the non-symmetric case, where in each of the two equations in $(1)_\tau$, the parameter $\tau$   is replaced  by two different parameters $\tau_1 > 0$ and $\tau_2 > 0$ respectively, and when also the second equation in $(1)_\tau$  includes a Dirac measure supported at the origin.
\end{abstract}



\section{Introduction}\label{intro}
\setcounter{equation}{0}

In recent years there has been a growing interest towards the understanding of non-abelian vortex configurations supported by Chern-Simons Gauge Field Theories, also in view of their connection to the delicate issue of "monopole confinement". See \cite {gud1, gud2} and references therein.

In this context, a successful way to detect vortices is to identify the BPS-sector of the theory, since in such regime vortex configurations simply correspond  to  (static) solutions of the so called \underline{self-dual equations} of Bogomolnyi type, and  saturate the minimal energy allowed by the system. For this reason, it has been useful to invoke "duality" and  formulate  the theory within the general framework of  $\mathcal{N}=2$ Supersymmetric (SUSY) Field Theory, in this direction see for example: \cite {efgknnov, efgno, gre, gud1, gud2, lmms, sy} for more details . 

We observe that, when the Chern-Simons Lagrangian is taken into account  then the theory can attain 
self-duality only with the help of a \underline{six-order scalar potential field}, see \cite{dunne, tar2, y}, in place of the more familiar quadratic (double-well) potential of the Maxwell-Higgs model, see \cite{jt}.  As a consequence, the system acquires additional vacua states, in the sense that now, the unbroken vacuum must coexist together with broken and partially broken vacua states. In turn, this feature give rise to new classes of self-dual vortices, each supported (asymptotically) by different vacua states.


On the other hand, in the non-abelian setting, it is quite difficult to handle self-dual vortices with mathematical rigour,  since even the reduced first order self-dual equations present serious analytical difficulties yet to be overcome. Nonetheless, it is possible to focus  on special classes of vortices satisfying some useful and physically consistent ansatz,  which permit to formulate the self-dual  Chern-Simons vortex equations into a form suitable for the approach in \cite{jt} to apply.  In this way, one is lead to the so called "Master equations", which govern such class of vortex configurations, and take the form of planar elliptic systems involving exponential nonlinearities and involve Dirac measures supported at the vortex points, see for example \cite{nt, ht, hlty, chls, tar3, y, dunne}. 


By means of the Master equations and their variational structure, it is possible to construct planar \underline{topological} vortices (asymptotically gauge equivalent to the \underline{unbroken} vacuum) by a minimisation procedure, see \cite{y, hlty, chls}. 
Similarly, after the variational approach introduced in \cite {cy, tar1} for the abelian case, it is possible to use variational methods also to establish multiple  \underline{periodic} non-abelian Chern-Simons vortices for various theories,  see e.g. \cite{nt, ht, hlty, chls}.\\

On the contrary,  the study of  planar \underline{non-topological vortices}, asymptotically gauge equivalent to the \underline{broken} vacuum, still remains unsatisfactory, in spite of the rather complete description available for the abelian case, see \cite{sy, cfl, ci, ckl}.  
Indeed, even when all the vortex points are superimposed, say at the origin, and one seeks radially symmetric vortices (about the origin), the corresponding \underline{radial} differential vortex problem turns out to be quite involved.
In this direction, we mention the contributions in \cite{hl1}, \cite{hl2}, \cite {ckl1} and \cite{chl} which provide useful general information about the nature of entire radial vortex solutions, and reveals their rich structure compared with the abelian situation. \\
On the basis of such results and as a way to progress further towards the understanding of  non-topological vortices, here we propose to adopt  the "perturbation"  strategy introduced by Chae-Imanuvilov in \cite{ci} for the abelian case,  see also  \cite{ct1, ct2} for non-abelian settings.

 
Actually,  to carry out such "perturbation" approach, the first important step is to provide rather accurate informations about the solvability of  Liouville - systems involving Dirac measures,  which occur as  "`limiting"' problems in the perturbation argument. 

Interestingly, the relevance of such class of  Liouville systems has emerged already in several other contexts, see \cite {ck2, k1,k2, k3, kl, csw, jow2} and references therein. Their study has concerned mainly the so called "cooperative" case, where all the entries of the coupling  matrix are assumed positive, and we refer to  \cite {ck1, ck2, csw, sw1, sw2, w, lz1, lz2, lz3, pot1, pot2}, for rather complete results about this situation. 


However, for the non-abelian models considered here, the corresponding Liouville systems are of "competitive" - type, and their solvability is far from understood, except for the very special case of the integrable Toda  system. We recall that, the  (m x m)-Toda system is characterised by its coupling matrix being the Cartan matrix of the (gauge) group $SU(m+1)$, and it describes vortices for a relativistic Chern-Simons model, see \cite{dunne}. \\
Lin-Wei-Ye in \cite{lwy} have identified  all the  (conformal) invariances enjoyed by the "singular" Toda system, and used them to obtain a full description of the corresponding solution set and of their linearised problem. The classification result in \cite{lwy} was motivated and inspired by the analogous classification result obtained by Jost-Wang \cite{jow2} for the  "regular"  Toda system, where the Dirac measures were neglected.\\
 Also note that for $m=1$ (corresponding to the abelian case  discussed in \cite{ci}), the result in \cite{lwy} just reduces to the well known classification result for solutions of "singular" Liouville equations obtained in \cite{pt} (see also \cite{cl1, cl2}), and corresponding linearised problem discussed in \cite{dem}.

The complete analysis provided by in \cite{lwy} for the singular Toda system has allowed  Ao-Lin-Wei in \cite{alw1, alw2}  to pursue the "`perturbation" approach of \cite{ci}, and construct non-topological vortices for the Chern-Simons model proposed by Dunne in \cite{dunne}, with any gauge group of rank 2.

The main purpose of this paper is to analyse the solvability of (not necessarily integrable) planar "`singular"' Liouville systems, as they occur in various Chern-Simons vortex problems, see e.g. \cite{gud1, gud2, lmms}, or in other physical context decribed for example in  \cite {k1, k2, k3, kl}, and which include the Toda system as a particular case. \\
In our situation, the systems enjoys only a scale invariance property, and so the admissible "flux-pair" corresponding to a solution is no longer a fixed value (as in the conformal / integrable case) but now it describes a curve, which we will be able to describe completely in the radially symmetric case. 

More precisely, we provide \underline{necessary and sufficient} conditions for the radial solvability of a general singular 2 X 2 Liouville system of Toda type,  in terms of the "flux pair" of a solution (given by certain integral values), which relates to the total energy carried by the corresponding vortex, see Theorem \ref{teoA} for the precise statement. In certain cases, our results are sharp also when the radial assumption is dropped, see Corollary \ref{corA2} . \\
In certain sense, our analysis should be considered  the continuation of that in \cite{pot2}  (see also \cite{pot1}) of  "cooperative" systems.\\

To simplify notation and  to better clarify  the arguments involved, we shall work mainly in the symmetric case, although it will be clear by our discussion how to modify the given arguments in order to treat non-symmetric systems as well. We mention also that, \cite{ft}, \cite{cht} and \cite{ggw} contain some particular cases of our result.\\
A central aspect of our proof is to provide an accurate blow up analysis for solutions of problem $(1)_\tau$, which relies in a crucial way upon the radial symmetric assumption. In fact, to obtain a similar blow-up description for general (non-radial) solutions would be much harder to attain, as we can see already for the integrable Toda system analysed first in \cite{jow1} and more recently in  \cite{lwz}  (see also references therein), or for a degenerate system given by a Cosmic Strings equation discussed in \cite{tar4}. \\
In the next section, we provide the necessary preliminaries and the precise statements of our results. 

\section{ Preliminaries and Statement of the Main Results }\label{main}  
\setcounter{equation}{0}

Motivated by the delicate issue of "`quark confinement"', Gudnason in \cite{gud1,gud2} introduced a non-abelian Chern-Simons model formulated within a  $\mathcal{N}=2$ Supersymmetric (SUSY) Field Theory, with a general gauge group of the type: $G=U(1) \times G'$ allowing solutions with orientational modes. For the model in \cite{gud1, gud2}, the author identifies the BPS-sector of the theory and the corresponding self-dual equations. In particular, when  $G' = SO(2)$ or $G' = U S_p (2)$, Gudnason in \cite{gud1,gud2}
introduced some meaningful physical ansatz on the structure of the vortex solutions, by which  (as in \cite{jt}) the corresponding self-dual equations reduced to the following set of  Master's equations: 
\begin{equation} \label{1.1} 
\begin{cases}
\Delta u = \alpha^2 (e^{u+v}+e^{u-v}-\xi)(e^{u+v}+e^{u-v})+\alpha \beta (e^{u+v}-e^{u-v})^2+ 4 \pi \sum_{j=1}^{n_1} \delta_{q_1,j}\\
\Delta v = \alpha \beta (e^{u+v}+e^{u-v}-\xi)(e^{u+v}-e^{u-v})+ \beta^2 (e^{2(u+v)}-e^{2(u-v)})^2+ 4 \pi \sum_{j=1}^{n_2} \delta_{q_2,j}
\end{cases}
\end{equation}

with $\alpha=\frac{\pi}{k_1}$ and $\beta=\frac{\pi}{k_2}$ and where $k_1$ and $k_2$ are respectively the coupling constants of the  $U(1)$ and the $G'$ part of the gauge field; moreover $\xi>0$ is known as the Fayet-Iliopoulus parameter.

The set of points (repeated with multiplicity):

\begin{equation}\label{12}
S_i = \{ q_{i,j}, j=1,..,n_i \} \quad i=1,2
\end{equation}  

where the Dirac measures are supported, are known as the \underline{vortex points}, and they carry a relevant physical meaning, see \cite{gud1, gud2}. 
In order to avoid singularities in the component $u-v$, we need to require that, 

\begin{equation}\label{13}
n_2 \leq n_1 \, \text{ and } \, S_2 \subset S_1.
\end{equation} 

For more details about the model and the derivation of \eqref{1.1} we refer to \cite{gud1,gud2} and \cite{hlty}.

Next, we observe that we can express \eqref{1.1}, in a form more familiar within the context of self-dual non-abelian Chern-Simons vortices, see \cite{dunne, tar3, y}.

To this purpose, we let,

\begin{equation}\label{14}
\lambda = \left( \xi \frac{\alpha}{2} \right)^2 \, \text{ and }\, k = \frac{\beta}{\alpha} = \frac{k_1}{k_2}
\end{equation}

and we introduce the new unknowns:

$$u_1= u + v \quad,\quad u_2 = u - v$$

in terms of which, the system \eqref{1.1} takes the form:

\begin{equation} \label{15} 
\begin{cases}
\Delta u_1 + \lambda^2 \left(\sum_{i=2}^{2} k_{1,i} e^{u_i} \left(1- \sum_{j=2}^{2} k_{i,j} e^{u_j} \right)\right) = 4 \pi \sum_{j=1}^{N_1} \delta_{p_1,j}\\
\Delta u_2 + \lambda^2 \left(\sum_{i=2}^{2} k_{2,i} e^{u_i} \left(1- \sum_{j=2}^{2} k_{i,j} e^{u_j} \right)\right) = 4 \pi \sum_{j=1}^{N_2} \delta_{p_2,j}
\end{cases}
\end{equation}

with coupling matrix $K=(k_{i,j})$ given as follows:

\begin{equation}\label{16}
K = \frac{1}{2} \left( \begin{array}{cc} 1+k & 1-k\\
                                  1-k & 1+k\\
																	\end{array} \right)
\end{equation}

and (by recalling \eqref{12}):

$$N_1 = n_1+n_2 \quad \text{ and }  \; \{ p_{1,j}, j=1,..,N_1 \}:= S_1 \cup S_2,$$
$$N_2 = n_1-n_2 \quad \text{ and } \; \{ p_{2,j}, j=1,..,N_2 \}:= S_1 \setminus S_2.$$

Observe that, for $\bf{k=3}$, the coupling matrix \eqref{16} coincides exactly with the Cartan matrix of the group SU(3) given as follows:

\begin{equation}\label{17}
K =  \left( \begin{array}{cc} 2 & -1\\
                                  -1 & 2\\
																	\end{array} \right)
\end{equation}

Therefore, in this case,  \eqref{15} reduces to the well known (integrable) 2 X 2 - Toda system, whose solutions describes $SU(3)$-vortices for the relativistic Chern-Simons- model proposed by Dunne \cite{dunne}. More generally, the model in \cite{dunne} is formulated in terms of any general (semi-simple) gauge group $G$,  with $\lambda = \left(\frac{1}{2 \mu}\right)^2$ and $\mu>0$ the Chern-Simons coupling parameter, see also \cite{y, tar3} for additional details. 

In the more general framework of the  $\mathcal{N}=2 -SUSY,$ we mention also the model proposed in \cite{lmms} concerning a non-abelian Chern-Simons theory with "flavors" and gauge group: $U(1)\times SU(N)$, where a system of the type \eqref{15} occurs as the  Master equations governing the corresponding self-dual vortex configurations  with coupling matrix:

\begin{equation}\label{18}
K = \frac{1}{N} \left( \begin{array}{cc} N-1+k & 1-k\\
                                  (N-1)(1-k) & 1+(N-1)k\\
                                																	\end{array} \right), \quad N>1
\end{equation}

where again $k=\frac{k_1}{k_2}$ and $k_1>0$, $k_2>0$ are respectively the coupling constants of the $U(1)$ and the $SU(N)$ part of the gauge fields, and $\lambda = \left(\frac{1}{2 k_1}\right)^2$.

Notice that, in all the given examples, and actually in other examples of physical interest, the corresponding coupling matrix $K$ is characterised  by the following property:

\begin{equation}\label{19}
k_{i,i}>0 \quad \text{ and } \quad k_{1,1}+k_{1,2} = k_{2,1}+k_{2,2} = 1.
\end{equation}

To construct planar vortices, we seek solutions of \eqref{15} in $\mathbb{R}^2$, satisfying the following integrability condition:

\begin{equation}\label{110}
\sum_{i=1}^{2} k_{l,i} e^{u_i} \left(1- \sum_{j=1}^{2} k_{i,j} e^{u_j} \right) \in L^1(\mathbb{R}^2),  \; l=1,2;
\end{equation}

which ensures that, the corresponding vortex admits \underline{ finite energy }, see \cite{dunne, gud1, gud2, lmms}.

By keeping in mind  \eqref{19}, from  \eqref{110} we identify the following suitable boundary conditions, $\text{ as } |x| \to \infty:$

\begin{equation}\label{111}
\text{ \underline{topological}: } \quad\quad  u_i (x) \to 0   \quad i=1,2,
\end{equation}

\begin{equation}\label{112}
\text{ \underline{non-topological}: } \quad \quad u_i (x) \to -\infty   \quad i=1,2,
\end{equation}

\begin{equation}\label{113}
\text{ \underline{mixed type}:}  \begin{array}{c} \quad \quad u_1 (x) \to -\infty \text{ and } u_2 (x) \to - \log{k_{2,2}},\\    
                                             \text{or}\\
                                              \quad \quad          u_2 (x) \to -\infty \text{ and } u_1 (x) \to - \log{k_{1,1}}.
\end{array}																												
\end{equation}

From the physical point of view, we mention that, topological solutions give rise to vortices asymptotically (at infinity) gauge equivalent to the unbroken vacuum state, non-topological solutions give rise to vortices asymptotically gauge equivalent to the broken vacuum state, and finally mixed-type solutions allow for vortices asymptotically gauge equivalent to partially broken vacua, see \cite{dunne}, \cite{y, tar2} for details.\\

To further substantiate the physical pertinence of the boundary conditions specified above, we observe that (when  \eqref{19} holds) every solution of \eqref{15} satisfying either one of the boundary conditions \eqref{111}, \eqref{112} and \eqref{113},  verifies:
 $u_i < 0 \text{ in } \mathbb{R}^2 \; \text{for every} \; i=1,2$,  consistently with the physical applications. \\ 
Furthermore, in the case of interest here, namely:
$$k_{i,i} > 0 \text{ and } k_{i,j} <  0 \text{ with }  i\ne j =1,2,$$
and when $p_{i,l} = 0$ for every $ l=1,...,N_i \; \text{ and } i=1,2,$  then Huang-Lin in \cite{hl1, hl2} have shown that every entire radially symmetric solution of  \eqref{15} must satisfy either one of the boundary conditions above,  and in particular the integrability condition \eqref{110} automatically  holds in this case.\\  

The construction of topological solutions for \eqref{15} has been carried out first by Yang in  \cite{y}, and more recently in \cite{hlty} and \cite{chls} for more general choices of coupling matrices including \eqref{16} and \eqref{18} as particular cases.  We also mention that the existence of multiple \underline{periodic} solutions for \eqref{15} was established first by Nolasco-Tarantello in \cite{nt} for the choice of coupling matrix \eqref{17}, and more recently in \cite {ht},  \cite{hlty} and \cite{chls} for more general choices including  \eqref{16} and \eqref{18} respectively.\\ 

In this paper, we focus on the construction of \underline{non-topological} solutions for \eqref{15}.

To this purpose, we recall what happens for the embedded abelian configurations, corresponding to the situation where,

\begin{equation*}
N=N_1=N_2 \text{ and (after relabelling) } p_{1,j} = p_{2,j} = p_j, \, j=1,..,N,
\end{equation*}
and
$$u_1=u_2=u$$
with $u$ a solution of the \underline{abelian} Chern-Simons vortex equation (cfr. \cite{hkp, jw}) :

\begin{equation}\label{114}
\Delta u + \lambda e^u (1-e^u) = 4 \pi \sum_{j=1}^{N} \delta_{p_j} \quad \text{ in } \mathbb{R}^2,
\end{equation}

(recall \eqref{19}). For a detailed study of \eqref{114} we refer to \cite{tar3}.

In particular, we know that a non-topological abelian vortex will correspond to a solution $u$ of \eqref{114} satisfying: $e^u \in L^1(\mathbb{R}^2)$. To construct such class of solutions, we let:
$$ u_\varepsilon (x) = u \left( \frac{x}{\varepsilon}\right) + 2 \log \frac{1}{\varepsilon}, \; \quad\, \varepsilon>0;$$

and observe that $ u_\varepsilon$ satisfies:

\begin{equation}\label{115}
-\Delta u_\varepsilon = \lambda e^{u_\varepsilon} - \lambda \varepsilon^2 e^{2 u_\varepsilon} - 4 \pi \sum_{j=1}^{N} \delta_{\varepsilon p_j} \quad \text{ in } \mathbb{R}^2,
\end{equation}
and
\begin{equation}\label{115tris}
\int_{\mathbb{R}^2} e^u = \int_{\mathbb{R}^2} e^{u_\varepsilon}.
\end{equation}

Therefore, (at least formally) non-topological solutions of  \eqref{114} can be sought as  "perturbation" of solutions of the following "limiting" problem, obtained  by letting $\varepsilon \to 0$ in \eqref {115}, namely:
 \begin{equation}\label{116}
-\Delta u = \lambda e^{u} - 4 \pi N \delta_{0} \quad \text{ in } \mathbb{R}^2,
\end{equation}
and,
\begin{equation}\label{117}
e^u \in L^1(\mathbb{R}^2). 
\end{equation}

Solutions of \eqref{116}, \eqref{117} are explicitly known (see \eqref{128} below, and  \cite{pt},  \cite{cl1,cl2}), and they were used by Chae-Imanuvilov in \cite{ci} together with suitable "perturbation" techniques to obtain a family of non-topological solutions for \eqref{114}, we refer to \cite{ci} and \cite {tar2} fo details.\\
Actually, the subsequent analysis in \cite{cfl} and \cite{ckl} has provided a rather complete description about the whole set of non-topological solutions for the single equation \eqref{114}.\\



Unfortunately, the analysis of non-topological solutions for the non-abelian vortex system \eqref{15} is far from complete, in spite of the results in \cite{ alw1, alw2, hl1, hl2, ckl1}. 

For example, one we can perform a similar scaling  for a solution $(u_1,u_2)$ of \eqref{15} as follows:
\begin{equation}\label{120bis}
u_{1, \varepsilon} (x) = u_1 \left( \frac{x}{\lambda\varepsilon}\right) + 2 \log \frac{1}{\varepsilon} \quad , \quad u_{2,\varepsilon} (x) = u_2 \left( \frac{x}{\lambda\varepsilon}\right) + 2 \log \frac{1}{\varepsilon}
\end{equation}

and realise that, in order to pursue a "`perturbation"' approach similar to \cite{ci}, it is important to understand the complete solvability of the following  \underline{singular Liouville system}:

\begin{equation} \label{120} 
\begin{cases}
-\Delta u_1 = k_{1,1} e^{u_1} + k_{1,2} e^{u_2} - 4 \pi N_1 \delta_{0}\\
-\Delta u_2 = k_{2,1} e^{u_1} + k_{2,2} e^{u_2} - 4 \pi N_2 \delta_{0}\\
e^{u_1}, e^{u_2} \in L^1 (\mathbb{R}^2)
\end{cases}
\end{equation}

which "formally" we derive  by letting  \ $\varepsilon \to 0$ in the system satisfied by $(u_{1,\varepsilon}, u_{2,\varepsilon})$.\\

The analysis of \eqref{120} is interesting in its own, since similar Liouville systems (in absence of Dirac sources, i.e. $N_1=N_2 = 0$) have emerged in other contexts, such for example in statistical mechanics, see \cite{k1, k2, k3, kl} and  \cite{ck1, ck2}, population dynamics and plasma physics, see \cite{csw, sw1, sw2} and references therein.

More precisely, concerning \eqref{120}, we address the question of identifying the pairs: $(\beta_1, \beta_2)$ for which problem \eqref{120} admits a solution satisfying:

\begin{equation}\label{121}
\beta_1 = \frac{1}{2\pi} \int_{\mathbb{R}^2} e^{u_1} \, dx, \quad \, \quad \beta_2 = \frac{1}{2\pi} \int_{\mathbb{R}^2} e^{u_2} \, dx. 
\end{equation}

This is a natural question to investigate also from the physical point of view, since the values $\beta_1$ and $\beta_2$ indicate the energy level of the corresponding vortex configuration.

We also notice that, in the context of radial solutions, problem \eqref{120}, \eqref{121} can be turned conveniently into a more standard boundary value problem. 

Actually, in the radial setting, we have a complete answer about the question of solvability for the so called \underline{cooperative systems}, even for general $m \times m-systems,$ where the coupling matrix is assumed to satisfy the following:

\begin{equation}\label{122}
K=(k_{i,j})_{i,j=1,..,m} \, \text{ is symmetric, irreducible with non-negative entries.}
\end{equation}

Solutions of cooperative systems in absence of  Dirac measures are always radially symmetric, see \cite{csw}, and they  have been investigated in  \cite{ck1,ck2, csw, lz1, lz2, sw1, sw2, w, jow1, jow2}, also in connection with sharp versions for systems of the Moser-Trudinger inequality, see \cite{w}, \cite{jow1}, or via duality, by invoking  sharp versions of the $log-HLS$ inequality for the Free Energy, (see \cite{csw,sw1, sw2}).\\

More recently, those results have been extended by Poliakovsky-Tarantello in \cite{pot1, pot2}, to  describe the sharp \underline{radial} solvability for "singular" Liouville  systems including the Dirac measures and with coupling matrix $K$ satisfying  \eqref{122}.\\
It should be noticed that in this case, non-radial solutions do occur, as exhibited in \cite{pot1}. 
Furthermore,  \cite{pot1, pot2} cover also the  \underline{degenerate} case, (unlike \cite{csw} and \cite{lz1, lz2}) which includes a class of Liouville-type equations arising (as a degenerate 2x2-system) in the study of self-dual Cosmic Strings, see \cite{pot1, pot2} and \cite{tar4} for details.\\

For problem \eqref{120}, \eqref{121} the results \cite{pot1, pot2} can be summarised as follows:

\begin{theo}[\cite{ck1, ck2, csw, pot2}]\label{teo1} Assume that,

\begin{equation}\label{123a}
k_{i,j}=k_{j,i} \ne 0,\text{ for } i \ne j \in \{ 1,2 \},
\end{equation}

and let $N_i > -1$, $i=1,2$. Then the following conditions: 

\begin{equation}\label{124a}
k_{1,1} \beta_{1}^{2} + k_{2,2} \beta_{2}^{2} + 2 k_{1,2} \beta_{1} \beta_{2} - 4 (N_1 + 1) \beta_1 - 4 (N_2 + 1) \beta_2 = 0
\end{equation}

\begin{equation}\label{124b}
\frac{k_{1,2}}{|k_{1,2}|} (k_{i,i} \beta_i - 4(N_i +1)) < 0, \quad i=1,2
\end{equation}

are \underline{necessary} for the existence of a \underline{radial} solution to \eqref{120}, \eqref{121}.\\

Furthermore, if (Cooperative Systems)

\begin{equation}\label{123b}
k_{i,i}\geq 0 \quad , \quad k_{i,j}=k_{j,i} > 0, \quad i \ne j \in \{ 1,2 \} 
\end{equation}

then \eqref{124a}, \eqref{124b} are also \underline{sufficient} for the radial solvability of \eqref{120}, \eqref{121}.\\
In addition, if $-1 < N \leq 0$ then every solution of \eqref{120}, \eqref{121} is indeed radially symmetric.
\end{theo}
\qed
\begin{remark}\label{rmk11}
Observe that, the case $k_{1,2}=k_{2,1}=0$ is not interesting, as \eqref{120} decouples into two singular Liouville equations of the type \eqref{116}, and it is solvable if and only if   $k_{i,i} \beta_i = 4 (N_i+1)$, $i=1,2$.
\end{remark}

The conditions: \eqref{124a}-\eqref{124b} rely in a crucial way on the scale invariance of problem \eqref{120}, \eqref{121} under the following transformation:

\begin{equation}\label{125}
u_{i,\lambda} (x) = u_i (\lambda x) + 2 \log \lambda, \quad i = 1,2 \text{ and } \lambda > 0;
\end{equation}
in the sense that,  
$$(u_1,u_2) \, \text{ satisfies } \eqref{120}, \eqref{121} \Longleftrightarrow (u_{1,\lambda}, u_{2,\lambda}) \, \text{ satisfies } \eqref{120}, \eqref{121}; \quad \forall \;  \lambda > 0.$$

The following uniqueness result holds:

\begin{theo}[\cite{lz1,pot2}]\label{teo2}
Assume  \eqref{123b}, then for every $(\beta_1,\beta_2)$ satisfying \eqref{124a}, \eqref{124b} there correspond a \underline{unique} radial solution of \eqref{120}, \eqref{121} modulo the transformation \eqref{125}.
\end{theo}

The above uniqueness result is based upon a non-degeneracy property for the linearised problem around a given radial solution, in the sense that (in the radial framework) it admits as only degeneracies those originated by the scale invariance \eqref{125}; we refer to \cite{lz1, pot2} for details. This fact extends what we know to happen for the decoupled system \eqref{120}, with $k_{1,2} = k_{2,1} = 0$.\\

Clearly, Theorem \ref{teo1} and Theorem \ref{teo2} apply to Gudnason' s model (with coupling matrix  \eqref{16} and gauge group $G=U(1)\times G'$), when the strength  $k_2$ of the $G'$-component of the gauge field overpowers the strength  $k_1$ of the $U(1)$-component,  that is: $k= \frac{k_1}{k_2} < 1.$  So, on the basis of the uniqueness (and non-degeneracy) property of radial solutions given by Theorem \ref{teo2}, in this case it should be  possible to use bifurcation techniques in order to obtain non-topological (radial)  solutions for the whole system \eqref{15}.\\

As already observed, in case $N_1>0$ or $N_2>0$, then we no longer expect every solution of \eqref{120}, \eqref{121} to be radially symmetric. Such a symmetry breaking phenomenon already occurs for the single Liouville equation \eqref{116}-\eqref{117}, where every solution $u$ must satisfy:
\begin{equation}\label{128a}
\int_{\mathbb{R}^2} e^{u} \, dx = 4(N+1),
\end{equation}

and (in complex notation) takes the following form:

\begin{equation}\label{128}
u(z) = \log \left[\frac{8(N+1)^2 \mu |z|^{2N}}{\lambda(1+\mu|z^{N+1}+b|^2)^2}\right]
\end{equation}

with  $\mu > 0$ and $b \in \mathbb{C}$ satisfying: $b \ne 0$ if and only if $N \in \mathbb{N}$ (see \cite{pt}). In particular, we see that only for $N \notin \mathbb{N}$, \underline{all} solutions to \eqref{116}-\eqref{117} are radially symmetric.\\

On the other hand, it seems reasonable to expect that \eqref{124a}, \eqref{124b} still provide sharp necessary and sufficient conditions for the solvability of \eqref{120}, \eqref{121}, also beyond the radial situation. 
However, while the Pohozaev-type identity \eqref{124a} remains valid for non-radial solutions of  \eqref{120}-\eqref{121}, as it relies only upon the scaling invariance \eqref{125} and the asymptotic behaviour of  solutions at infinity (see Proposition \ref{pro12}  below), it is less obvious to verify whether the condition \eqref{124b} holds also for non-radial solutions.

More precisely, by letting: 

\begin{equation}\label{129}
\beta_{i}^{\infty} = k_{i,i} \beta_i + k_{i,j} \beta_j, \quad i \ne j \in \{ 1,2 \}
\end{equation}

the following holds:
\begin{prop}\label{pro12}
Let $u_i \in L_{loc}^{1} (\mathbb{R}^2)$, $i=1,2$, and suppose that $(u_1,u_2)$ satisfies \eqref{120}, \eqref{121} (in the sense of distributions) for a given pair $(\beta_1,\beta_2).$ We have:\\
(i) \begin{eqnarray}
\label{130a} &u_i (x) = - (\beta_{i}^{\infty}+2 N_i) \log |x| + O(1), \quad \text{ for } |x| \geq 1\\ 
\label{130b} &r \partial_r u_i (x) \to - (\beta_{i}^{\infty}+2 N_i), \quad \partial_\theta u_i (x) \to 0 \quad \text{ as } r= |x| \to +\infty
\end{eqnarray}
with $\beta_{i}^{\infty}$ given in \eqref{129}, $i=1,2;$ and $(r,\theta)$ the polar coordinates in $\mathbb{R}^2$.\\
In particular, there holds:
\begin{equation}\label{132}
\beta_{i}^{\infty} = k_{i,i} \beta_i + k_{i,j} \beta_j > 2(N_i +1), \quad \forall i \ne j \in \{ 1,2 \}.
\end{equation}
(ii) If $K=(k_{i,j})_{i,j=1,2}$ satisfies:
\begin{equation}\label{133}
k_{1,2} \cdot k_{2,1} \geq 0,
\end{equation}
then the pair $(\beta_1,\beta_2)$ verifies the following identity:
\begin{equation}\label{134}
k_{1,1} |k_{2,1}| \beta_{1}^{2} + k_{2,2} |k_{1,2}| \beta_{2}^{2} +  2 k_{1,2} |k_{2,1}| \beta_{1} \beta_{2} - 4 (N_1 + 1) |k_{2,1}| \beta_1 - 4 (N_2 + 1) |k_{1,2}| \beta_2 = 0.
\end{equation}
\end{prop}
\qed




Clearly \eqref{134} reduces to \eqref{124a} when $K$ is symmetric, and in fact in such case, Proposition \ref{pro12} is essentially established in \cite{ck1, ck2,  csw}. For completeness we shall indicate its proof in the following section.\\
Furthermore, if one of the off diagonal entries of $K$ vanishes, (so that  \eqref{133} holds with equal sign), then \eqref{134} simply gives back what we already know about the single Liouville equation  \eqref{116}. 
Actually in this case, one is left to analyse a single planar equation, which can be casted as a weighted Liouville equation and in this way it can be handled as in \cite{lin}, \cite{pot1}, \cite{det}.\\

So the true interesting case to investigate occurs when $k_{1,2} \cdot k_{2,1} >0.$\\

\begin{remark}\label{rmk3}
Condition \eqref{132} does not appear for the  \underline{cooperative} system \eqref{123b}, since in this case it is a consequence of \eqref{124a} and \eqref{124b}. Indeed if by contradiction we assume for example that,

$$k_{1,1} \beta_1 + k_{1,2} \beta_2 \leq 2(N_1 +1),$$

then $2 k_{1,1} \beta_{1}^{2} + 2 k_{1,2} \beta_{1} \beta_2 \leq 4(N_1 +1) \beta_1$ and from \eqref{124a} we see that necessarily $k_{2,2}\beta_2 \geq 4(N_2 +1),$ in contradiction with \eqref{124b}.
\end{remark}








The main focus of this paper is to investigate the solvability of \eqref{120}-\eqref{121} in the more delicate "competitive" situation, where we assume:

\begin{equation}\label{134bis}
k_{i,i}>0 \quad \text{ and } \quad k_{i,j}<0 \quad \text{ for } i \ne j \in \{ 1,2\}
\end{equation}

as it occurs for \eqref{16} when $k=\frac{k_1}{k_2}>1$.\\
In particular \eqref{134bis} covers the 2  X  2 -Toda system  \eqref{17}, which is the \underline{only} situation where, so far,  we can describe  the whole solution set of \eqref{120}. Indeed, it has been shown in \cite{lwy}, that in this case the system gains full conformal invariance and, in analogy with the single Liouville equation \eqref{116}, (cfr. \cite{pt}),  it is possible to identify explicitly the corresponding solutions. In particular, we may conclude the following about the values  $\beta_1$ and $\beta_2$ in \eqref{121}:

\begin{equation}\label{134*}
 \eqref{120}-\eqref{121} \text {with K in  \eqref{17} is solvable if and only if }  \beta_1=\beta_2= 2(N_1 + N_2 + 2). 
\end{equation}

We refer to  \cite {jow2} and  \cite{lwy} for a detailed description of the   Toda system and its relevant conserved quantities.\\

Under the assumption \eqref{134bis}, the necessary condition \eqref{132} plays a central role towards the solvability of \eqref{120}-\eqref{121} and it implies the following:  

 \begin{coro}\label{cor2a}
If $k_{i,j}\leq 0$ for $i \ne j \in \{ 1,2 \}$, then
\begin{equation}\label{136}
k_{i,i}>0, \; \ i \in \{ 1,2 \} \quad \text{ and } \quad det\,( K )> 0,
\end{equation}
are \underline{necessary} conditions for the solvability of \eqref{120}, \eqref{121}.
In particular in the symmetric case, the solvability of \eqref{120}, \eqref{121} requires that $K$ is \underline{strictly positive definite}.
\end{coro}

From \eqref{136}, we see in particular that \eqref{134bis} and \eqref{123b} are truly complementary and fully account for \eqref{123a}.\\
In order to proceed further, we observe that under the assumption \eqref{134} we can make an useful normalisation, simply by setting:

\begin{equation}\label{137}
u_i (x) = v_i (x) + 2 N_i \log |x| - \log k_{i,i}, \quad i=1,2;
\end{equation}

and in terms of the new unknowns $(v_1,v_2)$, we obtain the following problem:

\begin{equation} \label{138} 
\begin{cases}
-\Delta v_1 = |x|^{2N_1} e^{v_1} -  \tau_1 |x|^{2N_2} e^{v_2}\\
-\Delta v_2 = |x|^{2N_2} e^{v_2} -  \tau_2 |x|^{2N_1} e^{v_1}\\
\frac{1}{2\pi} \int_{\mathbb{R}^2} |x|^{2N_i} e^{v_i} = \beta_i k_{i,i}, \quad i=1,2
\end{cases}
\end{equation}

with,

\begin{equation}\label{139}
\tau_1= -\frac{k_{1,2}}{k_{2,2}}>0 \quad \text{ and } \quad \tau_2= -\frac{k_{2,1}}{k_{1,1}}>0.
\end{equation}

\begin{remark}\label{rmk4}
For the coupling matrix given in \eqref{16}, problem \eqref{138} occurs with, 

\begin{equation}\label{139bis}
\tau_1= \tau_2= \frac{k-1}{k+1}>0, \;  \text{ for } k>1.
\end{equation}

In particular, the singular \underline{Toda system} (where $k=3$) leads to \eqref{138} with,

\begin{equation}\label{140}
\tau_1= \tau_2= \frac{1}{2} \quad \text{ and } \quad \frac{1}{2\pi} \int_{\mathbb{R}^2} |x|^{2N_1} e^{v_1} = \frac{1}{2\pi} \int_{\mathbb{R}^2} |x|^{2N_2} e^{v_2} = 4(N_1+N_2+2),
\end{equation}
see \eqref{134*} and \eqref{138}.
\end{remark}

To simplify notations and to minimise technicalities, in this paper we shall focus on the study of  \eqref{138} in the \underline{symmetric} case, where we assume:

\begin{equation}\label{141}
\tau_1= \tau_2 := \tau > 0,
\end{equation}

and we shall give indications on how to deal with the non-symmetric case: $\tau_1 \ne \tau_2$ and $\tau_1\tau_2>0$, which will be discussed in details in a forthcoming paper.

For the original system \eqref{120}, the validity of \eqref{141} amounts to assume the following stronger version of \eqref{134}:

\begin{equation}\label{142}
k_{i,i}>0 \quad \text{ and } \quad \frac{k_{i,j}}{k_{j,j}} = \frac{k_{j,i}}{k_{i,i}} \quad \text{ for } \ i \ne j \in \{ 1,2 \}.
\end{equation}

We mention that, a system of the type \eqref{120} for which \eqref{142} holds, is referred to as a \underline{collaborating system}, see \cite{sw1, sw2}.

Since our original motivation was to analyse Gudnason' s model, w.l.o.g. we can further assume $N_1 \geq N_2$; and again to reduce useless technicalities, we shall specify the following:

\begin{equation}\label{143}
N_1=N>0 \quad \text{ and } \quad N_2=0
\end{equation}

Therefore, from now on we shall focus on the solvability of the following problem:

\begin{equation*} 
(P)_\tau \begin{cases}
-\Delta v_1 = |x|^{2N} e^{v_1} -  \tau  e^{v_2}\\
-\Delta v_2 = e^{v_2} -  \tau  |x|^{2N} e^{v_1}\\
\beta_1 = \frac{1}{2\pi} \int_{\mathbb{R}^2} |x|^{2N} e^{v_1}, \;  \; \beta_2 = \frac{1}{2\pi} \int_{\mathbb{R}^2} e^{v_2}, 
\end{cases}
\end{equation*}

with $\tau>0$ and $N>0$.

According to the discussion above, already we may claim the following about the solvability of $(P)_\tau$:

\begin{coro}\label{cor2}
If $\tau>0$ and $ \beta_i > 0, \; i=1,2,$  then the following conditions are \underline{necessary} for the solvability of $(P)_\tau$:

\begin{eqnarray}
\label{144} &\tau \in (0,1)&\\
\label{145} &\beta_{1}^{2} + \beta_{2}^{2} - 2 \tau \beta_{1} \beta_{2} - 4(N+1) \beta_1 - 4 \beta_2 = 0,&\\
\label{146} &\beta_1 - \tau \beta_2 > 2(N+1) \quad \text{ and } \quad \; \beta_2 - \tau \beta_1 > 2.&
\end{eqnarray}
In addition the condition:
\begin{equation}\label{147}
\beta_1>4(N+1) \quad  \text{ and  } \quad \beta_2 >4, 
\end{equation}
is \underline{necessary} for the existence of a \underline{radial} solution.
\end{coro}
\qed

In contrast to the case where $\tau<0$ (see part b) of Theorem \ref{teo1}), we do not expect in general that conditions \eqref{144}-\eqref{146} (and \eqref{147}) are also sufficient for the (radial) solvability of $(P)_\tau$. 

Indeed, it suffices to recall what happens for the Toda system, where by \eqref{140} and \eqref{143} we have:  
\begin{equation}\label{147*}
\text{ if } \tau=\frac12  \text{ then } \beta_1=\beta_2=4(N+2), 
\end{equation}
and we see clearly that \eqref{147*} is a much more stringent condition than \eqref{144}-\eqref{147}.\\

On the other hand, by virtue of part a) in Theorem \ref{teo1}, which  applies to $(P)_\tau$, and by keeping in mind part b) (which applies to $(P)_\tau$ only when $\tau<0$), by way of continuity, we expect that \eqref{145} and \eqref{147} should still provide necessary and sufficient conditions for the \underline{radial} solvability of $(P)_\tau$, when $\tau>0$ is sufficiently small.

We establish this fact and prove the following:

\begin{theo}\label{teo3}
Let $N>0$, then for any,

\begin{equation}\label{148}
0<\tau\leq \frac{1}{N+1 + \sqrt{(N+1)^2+4}} := \tau_0 (N)
\end{equation}

conditions \eqref{145} and \eqref{147} are necessary and sufficient for the radial solvability of $(P)_\tau$.
\end{theo}
\qed

Clearly $\tau_0 (N) < \frac12$ (recall that $\tau=\frac12$ corresponds to the Toda system) and, as we shall see, it corresponds to the sharp value for which conditions \eqref{145}, \eqref{147} imply \eqref{146}, consistently with Remark \ref{rmk3}.

By combining Theorem \ref{teo1} and \ref{teo3} we can conclude the following about problem $(P)_\tau$.

\begin{coro}\label{cor3}
For every $0 \ne \tau \leq \tau_0 (N)$ the conditions \eqref{145} and 

\begin{equation}\label{149}
\frac{\tau}{|\tau|} (\beta_1 - 4 (N+1))<0 \quad ; \quad \frac{\tau}{|\tau|} (\beta_2 - 4)<0
\end{equation}

are necessary and sufficient for the \underline{radial} solvability of $(P)_\tau$.
\end{coro}
\qed

We suspect that, also the uniqueness (and non-degeneracy) result of Theorem \ref{teo2}, should remain valid for radial solutions of $(P)_\tau$, whenever $\tau < \tau_0 (N)$, and not only when $\tau \leq 0$.\\

Interestingly, for $\tau$ sufficiently "close" to $1$, the conditions \eqref{145} and \eqref{146} are the ones responsible for the solvability of $(P)_\tau$.

To be more precise we consider the following  polynomial, whose role will become clear in the following,

$$\psi_1(\tau) = 2 (1-2\tau^2)(1+2\tau(N+1))-1$$
and by direct inspection we check that:
\begin{equation}\label{150}
\begin{split}
&\psi_1(\tau) \text{ admits a unique positive zero, denoted by } \tau_1 = \tau_1 (N), \\
&\text{ and } \tau_1 \in \left(\frac{1}{2},\frac{1}{\sqrt{2}}\right)
\end{split}
\end{equation}

\begin{theo}\label{teo4}
For any $\tau \in [\tau_1,1)$, with $\tau_1 = \tau_1 (N)$ defined in \eqref{150}, we have that the conditions \eqref{145} and  \eqref{146} are necessary and sufficient for the solvability of $(P)_\tau$.
\end{theo}
\qed

Actually, under the assumptions of Theorem \ref{teo4}, we shall see that the conditions \eqref{145} and  \eqref{146} imply  \eqref{147} and permit to establish the existence of a radial solution for $(P)_\tau$. 
Hence, it is tempting to  conjecture that $(P)_\tau$ may  admit \underline{only} radial solutions in this case.\\

The case where $\tau \in (\tau_0, \tau_1)$ is more delicate, as it includes the Toda system ($\tau=1/2$), where we know that the necessary conditions \eqref{145}-\eqref{147} are no longer sufficient for solvability, see \eqref{147*}.


In fact, for $\tau \in (\tau_0, \tau_1)$ we are able to identify new necessary and sufficient conditions for the (radial) solvability of $(P)_\tau$, which indeed reduce to \eqref{147*} when $\tau=1/2$.

To this end we observe that the conditions \eqref{145} and  \eqref{146} define a graph in the first quadrant of the $(\beta_1,\beta_2)$-plane.

More precisely, for $\tau \in (0,1)$, we let:

\begin{equation}\label{152a}
\begin{split}
&\underline{\beta}_{1}:=\underline{\beta}_{1}(\tau)=\dfrac{2}{1-\tau^2}\left(N+1+\tau+\tau\sqrt{(N+1)^{2}+2\tau(N+1)+1}\right)\\
&\overline{\beta}_{1}:=\overline{\beta}_{1}(\tau)=\dfrac{2}{1-\tau^2}\left(N+1+\tau+\sqrt{(N+1)^{2}+2\tau(N+1)+1}\right) > 4(N+1)
\end{split}
\end{equation}

so that,  $\underline{\beta}_{1}(\tau) < \overline{\beta}_{1}(\tau)$; and similarly, we let:

\begin{equation}\label{152b}
\begin{split}
&\underline{\beta}_{2}:=\underline{\beta}_{2}(\tau)=\dfrac{2}{1-\tau^2}\left(1+\tau(N+1)+\tau\sqrt{(N+1)^{2}+2\tau(N+1)+1}\right)\\
&\overline{\beta}_{2}:=\overline{\beta}_{2}(\tau)=\dfrac{2}{1-\tau^2}\left(1+\tau(N+1)+\sqrt{(N+1)^{2}+2\tau(N+1)+1}\right) > 4
\end{split}
\end{equation}

so that, $\underline{\beta}_{2}(\tau) < \overline{\beta}_{2}(\tau)$.

We easily check that: $(\beta_1, \beta_2)$ satisfies \eqref{145} and \eqref{146} if and only if, 

\begin{equation*}
\beta_1 \in (\underline{\beta}_{1},\overline{\beta}_{1}) \text{ and } \beta_2 = 2 + \tau \beta_1 + \sqrt{(2 + \tau \beta_1)^2 - \beta_1 (\beta_1 - 4 (N+1))}:= \varphi_{1}^{+} (\beta_1),
\end{equation*}

or equivalently:

\begin{equation}\label{153}
\beta_2 \in (\underline{\beta}_{2},\overline{\beta}_{2}) \text{ and } \beta_1 = 2(N+1) + \tau \beta_2 + \sqrt{(2 (N+1)+ \tau \beta_2)^2 - \beta_2 (\beta_2 - 4 )}:= \varphi_{2}^{+} (\beta_2).
\end{equation}

We prove the following:

\begin{theo}\label{teo5}
If  $\frac12 \ne \tau \in (\tau_0,\tau_1)$ and  we let $\beta_{i}^{\pm} (\tau)$, $i=1,2,$ defined in \eqref{613} and \eqref{613*} below,  then the following hold:

\begin{equation}\label{154}
\begin{split}
&\max\{ 4(N+1), \underline{\beta}_{1}(\tau)\} \leq \beta_{1}^{-} (\tau) < \beta_{1}^{+} (\tau) \leq \overline{\beta}_{1}(\tau), \\
&\max\{ 4, \underline{\beta}_{2}(\tau)\} \leq \beta_{2}^{-} (\tau) < \beta_{2}^{+} (\tau) < \overline{\beta}_{2}(\tau); 
\end{split}
\end{equation}

moreover problem $(P)_\tau$ admits a radially symmetric solution if and only if the pair $(\beta_1,\beta_2)$ satisfies:

\begin{equation}\label{155a}
\beta_1 \in (\beta_{1}^{-} (\tau),\beta_{1}^{+} (\tau)) \text{ and } \beta_2 = \varphi_{1}^{+} (\beta_1),
\end{equation}

or equivalently:

\begin{equation}\label{155b}
\beta_2 \in (\beta_{2}^{-} (\tau),\beta_{2}^{+} (\tau)) \text{ and } \beta_1 = \varphi_{2}^{+} (\beta_2).
\end{equation}

Furthermore,
\begin{equation}\label{155b}
\beta_{i}^{\pm} (\tau) \to 4(N+2) \text{ as } \tau \to \frac{1}{2}.
\end{equation}
\end{theo}
\qed

We refer to Theorem \ref{teoA},  Theorem \ref{teoB} and Theorem \ref{teoC}  for more detailed statements and a better grasp about the origin of the value: $\beta_{i}^{\pm}$, $i=1,2$.\\


To establish the results stated above, we use a blow up analysis.

More precisely, for every $\alpha \in \mathbb{R}$ we let $(v_1(r,\alpha), v_2 (r,\alpha))$  the solution of the Cauchy problem for the radial ODE system corresponding to $(P)_\tau$ with initial conditions:

\begin{equation}\label{156}
v_1(0,\alpha) = \alpha,  \quad v_2 (0,\alpha)= 0, \quad \dot{v}_1 (0,\alpha) =\dot{v}_2 (0,\alpha)=0. 
\end{equation}

We check that such solution is unique and globally defined for every $r \geq 0$, and it satisfies the required integrability conditions: $|x|^{2N}e^{v_1} \in L^1 (\mathbb{R}^2)$ and $e^{v_2} \in L^1 (\mathbb{R}^2).$

Thus by setting:

\begin{equation}\label{157}
\beta_1 (\alpha) := \int_{0}^{+\infty} r^{2N+1} e^{v_1(r)} \, dr \quad \text{ and } \quad \beta_2 (\alpha) := \int_{0}^{+\infty} r e^{v_2(r)} \, dr
\end{equation}

we can use the scale invariance \eqref{125}, in order to see that the whole set of pairs $(\beta_1,\beta_2)$ for which $(P)_\tau$ admits a radial solution is fully described by the smooth curve:

\begin{equation}\label{158}
(\beta_1 (\alpha),\beta_2 (\alpha)), \alpha \in \mathbb{R}.
\end{equation}

By using blow-up techniques introduced in \cite {bm, bt, ls} in the context of Liouville-type equations, we shall be able to identify the limit value of $\beta_i (\alpha)$, as $\alpha \to +\infty$ and $\alpha \to -\infty$ for $i=1,2,$  and we check that indeed they yield to the (sufficient) statement about the existence of radial solutions claimed in Theorem \ref{teo1}, Theorem \ref{teo2} and Theorem \ref{teo3}. The "necessary" part is contained in Theorem \ref{teoC}\\
It is worth to notice that the "limit" values of $\beta_i (\alpha)$  in general do not capture the sharp bounds for the pairs $(\beta_1, \beta_2)$,  and this fact has already come up for a single equation Liouville-type , as discussed in \cite {lin}, and \cite{det}. So our  classification result gives also indication of certain non-degeneracy properties enjoyed by radial solutions of problem $(P)_\tau$,  to be further investigated.

\section{Some Useful Facts}
\setcounter{equation}{0}
In this section, we collect some general informations about solutions of \eqref{120}, \eqref{121}, in terms of the coupling matrix $K$ and the "flux-pair" $(\beta_1, \beta_2),$ of independent interest. They generalise some known facts concerning the 
symmetric non degenerate case, see \cite{ck1, ck2, csw}.

As above, we let:

\begin{equation}\label{20}
u_i (x) = v_i (x) + 2 N_i \log |x|, \; i=1,2
\end{equation}

and formulate \eqref{120}, \eqref{121} in terms of $(v_1,v_2)$ (the regular parts of $(u_1,u_2)$), as follows:

\begin{equation}\label{21}
\begin{cases}
-\Delta v_1 = k_{1,1} |x|^{2N_1} e^{v_1} + k_{1,2} |x|^{2N_2} e^{v_2} \quad \text{ in } \mathbb{R}^2\\
-\Delta v_2 = k_{2,1} |x|^{2N_1} e^{v_1} + k_{2,2} |x|^{2N_2} e^{v_2} \quad \text{ in } \mathbb{R}^2\\
\beta_1 = \frac{1}{2\pi} \int_{\mathbb{R}^2} |x|^{2N_1} e^{v_1} \; , \; \beta_2 = \frac{1}{2\pi} \int_{\mathbb{R}^2} |x|^{2N_2} e^{v_2} 
\end{cases}
\end{equation}
The following holds:
\begin{lemma}\label{lem1}
If $(v_1,v_2)$ satisfies \eqref{21} then we have:

\begin{equation}\label{22} 
\begin{split}
&v_i (x) = - \beta_{i}^{\infty} \log |x| + O(1), \quad \text{ for } |x| \geq 1\\ 
&r \partial_r v_i (x) \to - \beta_{i}^{\infty}, \quad \partial_\theta v_i (x) \to 0 \quad \text{ as } r= |x| \to +\infty
\end{split}
\end{equation}
with $\beta_{i}^{\infty}$ in \eqref{129},  $i=1,2;$ and $(r,\theta)$ polar coordinates.
In particular the condition:
\begin{equation}\label{22bis}
\beta_{i}^{\infty} = k_{i,i} \beta_i + k_{i,j} \beta_j > 2(N_i +1), \quad \forall i \ne j \in \{ 1,2 \}
\end{equation}
is \underline{necessary} for the solvability of \eqref{21}.\\
Furthermore, \\
i) if $k_{i,j}\leq 0$ for $i \ne j \in \{ 1,2 \}$ then necessarily the coupling matrix $K$ must satisfy \eqref{136}.\\
ii) If $k_{i,j}> 0$ for $i \ne j \in \{ 1,2 \}$ but $k_{i,i}\leq 0$ for some $i \in \{ 1,2 \}$, then necessarily $det K < 0$.
\end{lemma}
\proof: By an obvious modification of the arguments provided in \cite{cl1,cl2}, one can derive that every solution $(v_1,v_2)$  of \eqref{21} must satisfy: $v_{i}^{+}  \in L^{\infty} (\mathbb{R}^2), \quad  i=1,2.$
Hence, can use such information and suitable potential estimates, as in \cite{cl1, cl2} or in \cite{ct2}, in order to deduce \eqref{22}. At this point, we can obtain \eqref{22bis} as an easy consequence of \eqref{22} and the integrability of the function $|x|^{2N_i} e^{v_i}$ in $\mathbb{R}^2,$  $i=1,2.$






Concerning $i)$, we observe that if $k_{i,j} = 0$ for some $i \ne j \in \{ 1,2 \}$ then \eqref{136} follows directly from \eqref{22bis}.  While if $k_{i,j} < 0, \, \forall i \ne j \in \{ 1,2 \},$ then from \eqref{22bis} we see that $ k_{i,i}>0$ for all $i \in \{ 1,2 \}$. Furthermore, we can rewrite \eqref{22bis} equivalently as follows:
\begin{eqnarray}\label{135}
\notag (k_{1,1} k_{2,2} - k_{1,2} k_{2,1}) \beta_1 > 2(N_1+1) k_{2,2} + 2(N_2+1) |k_{1,2}|\\
\label{135} (k_{1,1} k_{2,2} - k_{1,2} k_{2,1}) \beta_2 > 2(N_2+1) k_{1,1} + 2(N_1+1) |k_{2,1}|.
\end{eqnarray}
and  \eqref{136} readily follows. Similarly, to deduce $ii)$ we assume for example that $k_{1,1}\leq 0$, then from \eqref{22bis} we find that,
\begin{equation}\label{22bisbis}
-\beta_2(det K)= \beta_2( k_{2,2}|k_{1,1}| + k_{1,2}k_{2,1}) > 2(N_1 + 1)k_{2,1} + 2(N_2 + 1) |k_{1,1}|> 0,
\end{equation}
and necessarily $det K < 0$ as claimed.

\qed

\begin{lemma}\label{lem1a}
Let $K = (k_{i,j})_{i,j=1,2}$ satisfy:

\begin{equation*}
k_{12} \cdot k_{21} \geq 0 .
\end{equation*}

If problem \eqref{12} admits a solution then the pairs $(\beta_1,\beta_2)$ must satisfy:

\begin{equation}\label{23}
k_{1,1} |k_{2,1}| \beta_{1}^{2} + k_{2,2} |k_{1,2}| \beta_{2}^{2} + 2 k_{1,2} |k_{2,1}| \beta_{1} \beta_{2} - 4 (N_1 + 1) |k_{2,1}| \beta_1 - 4 (N_2 + 1) |k_{1,2}| \beta_2 = 0.
\end{equation}

\end{lemma}

\proof: In case $k_{1,2}=0=k_{2,1}$ then the left hand side of \eqref{23} is identically zero and we have nothing to prove. On the other hand, the system in this case decouples into two singular Liouville equations, whose solvability requires $k_{i,i}>0$ for every $i=1,2$ and whose solution set is completely described in \cite{pt}.

Furthermore if $k_{1,2}=0$ but $k_{2,1} \ne 0$ (or $k_{2,1}=0$ but $k_{1,2}\ne0$) then only the first (or second) equation becomes a singular Liouville equation and \eqref{23} simply leads to the following well known fact,

\begin{equation*}
k_{1,1} \beta_1 - 4 (N_1+1)=0 \quad ( \text{ or } k_{2,2} \beta_2 - 4 (N_2+1)=0)
\end{equation*}

see \cite{cl2,pt}.

So the true interesting situation to analyse occurs when, 

\begin{equation}\label{23bis}
k_{1,2} \cdot k_{2,1} > 0.
\end{equation}

\textbf{Claim 1:} If $k_{1,2}=k_{2,1} \ne 0$  then \eqref{23} holds, or equivalently:

\begin{equation}\label{24}
k_{1,1} \beta_{1}^{2} + k_{2,2} \beta_{2}^{2} + 2 k_{1,2} \beta_{1} \beta_{2} - 4 (N_1 + 1) \beta_1 - 4 (N_2 + 1) \beta_2 = 0.
\end{equation}

It is clear that \eqref{23} and \eqref{24} are equivalent when $k_{1,2}=k_{2,1} \ne 0$.

To establish \eqref{24} we use Pohozaev' s trick and multiply the first equations in \eqref{21} by $\nabla v_1 (x) \cdot x,$ and the second equation in \eqref{21} by $\nabla v_2 (x) \cdot x,$ and integrate over the ball $B_r = \{ x \in \mathbb{R}^2 : |x|<r \}$ to obtain the following identity:

\begin{equation}\label{25}
\begin{split}
&\int_{\partial B_r}r\left(\dfrac{1}{2}\left|\nabla v_i\right|^{2}-\left(\dfrac{\partial v_i}{\partial \nu}\right)^{2}\right)d\sigma=k_{i,i} \int_{\partial B_r}r^{2N_{i}+1}e^{v_i} d\sigma\\
&-2(N_{i}+1)k_{i,i} \int_{B_r}\left|x\right|^{2N_{i}}e^{v_i}+k_{i,j}\int_{B_r}\left|x\right|^{2N_{j}} e^{v_j} \nabla v_i (x) \cdot x, \quad i \ne j \in \{ 1,2 \}
\end{split}
\end{equation}

where we have used the well known identity:

\begin{equation*}
(\Delta v (x)) \nabla v(x) \cdot x = div( \nabla v(x) (x \cdot v(x) - x \frac{|\nabla v|^2}{2})
\end{equation*}

and Green-Gauss theorem.

On the other hand, if we multiply the first equation in \eqref{21} by $\nabla v_2 (x) \cdot x,$ and the second equation in \eqref{21} by $\nabla v_1 (x) \cdot x,$ and then we integrate over $B_r$, we find:

\begin{equation}\label{26}
\begin{split}
&-\int_{\partial B_r}\Delta v_i (\nabla v_j \cdot x)dx = k_{i,i} \int_{B_r}\left|x\right|^{2N_{i}} e^{v_i} (\nabla v_j (x) \cdot x)dx\\
&+k_{i,j}\int_{\partial B_r}r^{2N_{j}+1} e^{v_j} \, d\sigma -2(N_{j}+1)k_{i,j} \int_{B_r}\left|x\right|^{2N_{j}}e^{v_j}, \quad i \ne j \in \{ 1,2 \}
\end{split}
\end{equation}

where again we have used the Green-Gauss theorem.

Since $k_{1,2}=k_{2,1}$, from \eqref{25} and \eqref{26} we derive the following identities:

\begin{equation}\label{27}
\begin{split} 
&k_{1,2} \int_{B_r} (\Delta v_1 (x) \nabla v_2 (x) \cdot x + \Delta v_2 (x) \nabla v_1 (x) \cdot x)dx = \\
&- k_{1,2} k_{1,1} \int_{B_r} |x|^{2N_1} e^{v_1 (x)} (\nabla v_2 (x) \cdot x)dx - k_{1,2} k_{2,2} \int_{B_r} |x|^{2N_2} e^{v_2 (x)} (\nabla v_1 (x) \cdot x)dx \\
&- k_{1,2}^2 \int_{\partial B_r} r^{2N_1+1} e^{v_1} \, d\sigma - k_{1,2}^2 \int_{\partial B_r} r^{2N_2+1} e^{v_2} \, d\sigma\\
&+ k_{1,2}^2 \left( 2(N_1+1) \int_{B_r} |x|^{2N_1} e^{v_1}dx + 2 (N_2+1) \int_{B_r} |x|^{2N_2} e^{v_2 (x)}dx \right)\\
\end{split}
\end{equation}

and

\begin{equation}\label{28}
\begin{split} 
&k_{2,2} \int_{\partial B_r}r \left(\dfrac{1}{2}\left|\nabla v_1\right|^{2}-\left(\dfrac{\partial v_1}{\partial \nu}\right)^{2}\right)d\sigma + k_{1,1} \int_{\partial B_r}r \left(\dfrac{1}{2}\left|\nabla v_2\right|^{2}-\left(\dfrac{\partial v_2}{\partial \nu}\right)^{2}\right)d\sigma = \\
&- k_{1,1} k_{2,2} \left[ \int_{\partial B_r} r^{2N_1+1} e^{v_1} \, d\sigma + \int_{\partial B_r} r^{2N_2+1} e^{v_2} \, d\sigma \right.\\
&\left. -2(N_1+1) \int_{B_r} |x|^{2N_1} e^{v_1}dx - 2 (N_2+1) \int_{B_r} |x|^{2N_2} e^{v_2 (x)}dx \right]\\
&+k_{1,2} k_{1,1} \int_{B_r} |x|^{2N_1} e^{v_1 (x)}( \nabla v_2 (x) \cdot x)dx + k_{1,2} k_{2,2} \int_{B_r} |x|^{2N_2} e^{v_2 (x)}( \nabla v_1 (x) \cdot x)dx  \\
\end{split}
\end{equation}

By recalling the identity:

\begin{equation}\label{29}
\begin{split} 
&\Delta v_1 (x) (\nabla v_1 (x) \cdot x) + \Delta v_2 (x) (\nabla v_2 (x) \cdot x )= div (x \cdot \nabla v_1 \nabla v_2 - x \cdot \nabla^{\perp} v_{1} \nabla^{\perp} v_{2}) =\\
&= div (x \cdot \nabla v_2 \nabla v_1 - x \cdot \nabla^{\perp} v_{2} \nabla^{\perp} v_{1}),
\end{split}
\end{equation}


where: $\nabla^{\perp} v = \left( \frac{\partial v}{\partial x_2}, - \frac{\partial v}{\partial x_1}\right)$,
we can turn also the right hand side of \eqref{27} into a boundary integral, and by summing up \eqref{27} and \eqref{28} we arrive at the following:

\begin{equation}\label{211}
\begin{split} 
&k_{1,1}  \int_{\partial B_r} r\left(\dfrac{1}{2}\left|\nabla v_2\right|^{2}-\left(\dfrac{\partial v_2}{\partial \nu}\right)^{2}\right)d\sigma + k_{2,2} \int_{\partial B_r} r \left(\dfrac{1}{2}\left|\nabla v_1\right|^{2}-\left(\dfrac{\partial v_1}{\partial \nu}\right)^{2}\right)d\sigma  \\
&+ k_{1,2}  \int_{\partial B_r} r \left(\dfrac{\partial v_1}{\partial \nu} \dfrac{\partial v_2}{\partial \nu}  - (\nabla^{\perp} v_{1} \cdot \nu) (\nabla^{\perp} v_{2} \cdot \nu) \right)=\\
&(det K) \left( \int_{\partial B_r} r^{2N_1+1} e^{v_1} \, d\sigma +  \int_{\partial B_r} r^{2N_2+1} e^{v_2} \, d\sigma \right)\\
&-(det K) \left( 2(N_1+1) \int_{B_r} |x|^{2N_1} e^{v_1} + 2 (N_2+1) \int_{B_r} |x|^{2N_2} e^{v_2 (x)} \right).\\
\end{split}
\end{equation}

At this point, by taking into account the asymptotic behaviour of $v_1$ and $v_2$ as given in \eqref{22} and \eqref{22bis}, we can pass to the limit, as $r \to \infty$ in \eqref{211}, and conclude:

\begin{equation}\label{212}
k_{1,2} \beta_{1}^{\infty} \beta_{2}^{\infty} - \frac12 k_{2,2} (\beta_{1}^{\infty})^2 - \frac12 k_{1,1} (\beta_{2}^{\infty})^2 = - (det K) (2(N_1+1) \beta_1) + 2(N_2+1) \beta_2) 
\end{equation}

with,

\begin{equation}\label{213}
\beta_{1}^{\infty} = k_{1,1} \beta_{1} + k_{1,2} \beta_{2} \, \text{ and } \, \beta_{2}^{\infty} = k_{1,2} \beta_{1} + k_{2,2} \beta_{2}, 
\end{equation}

see \eqref{129}. Identity \eqref{212} reaffirms i) and ii) of Lemma \ref{lem1}.

By inserting \eqref{213} into \eqref{212} and by carrying out straightforward calculations, we easily arrive at \eqref{24}, provided that: $det K \ne 0$.

On the other hand, by Lemma \ref{lem1}, we also know that $det K \ne 0$ enters as a necessary condition for the solvability of \eqref{21}, except in the \underline{cooperative} case \eqref{123b} where $K$ could be degenerate.

But for \underline{degenerate} cooperative systems, the matrix  $K$ satisfies:

\begin{equation}\label{216.1}
k_{i,j}>0 \quad i,j \in \{ 1,2 \} \, \text{ and } k_{1,1} k_{2,2} = k_{1,2}^{2}, 
\end{equation}

and so necessarily,

\begin{equation}\label{216.2}
\beta_{2}^{\infty} = \frac{k_{1,2}}{k_{1,1}} \beta_{1}^{\infty} \quad \text{ and } \quad v_{2} = \frac{k_{1,2}}{k_{1,1}} v_{1}.
\end{equation}

In other words, by setting:

\begin{equation}\label{216.3}
\beta = \beta_{1}^{\infty}  \quad v= v_{1} \quad \text{ and } \, a = \frac{k_{1,2}}{k_{1,1}},
\end{equation}

we see that in this case \eqref{21} reduces to the following single Liouville type equation:

\begin{equation}\label{216.4}
\begin{cases}
-\Delta v = k_{1,1} |x|^{2N_1} e^{v} + k_{1,2} |x|^{2N_2} e^{a v} \quad \text{ in } \mathbb{R}^2\\
\beta = \frac{1}{2\pi} \int_{\mathbb{R}^2}( k_{1,1} |x|^{2N_1} e^{v} + k_{1,2} |x|^{2N_2} e^{a v})dx 
\end{cases}
\end{equation}

with $k_{1,1}>0$,  $k_{1,2}>0$ and $k_{1,1}k_{2,2}=k_{1,2}^2$.

Equations of this type arise in the construction of Self-gravitating Cosmic Strings (cfr. \cite{y}), and have been analysed in \cite{cgs, pot1, pot2, tar4}.

In particular, in this case, for $i=1$ we may complete \eqref{25},  by using \eqref{216.2} and \eqref{216.3} and obtain that every solutions of \eqref{216.4} satisfies:

\begin{equation*}
\begin{split} 
& r \int_{\partial B_r}\left(\dfrac{1}{2}\left|\nabla v\right|^{2}-\left(\dfrac{\partial v}{\partial \nu}\right)^{2}\right)d\sigma = k_{1,1} \int_{\partial B_r} r^{2N_1+1} e^{v} \, d\sigma + \frac{k_{1,2}}{a} \int_{\partial B_r} r^{2N_2+1} e^{a v} \, d\sigma\\
&-2(N_1+1) k_{1,1} \int_{B_r} |x|^{2N_1} e^{v} - k_{1,2} \frac{2 (N_2+1)}{a} \int_{B_r} |x|^{2N_2} e^{a v}.
\end{split}
\end{equation*}

Therefore by passing to the limit in the above identity as $r \to \infty$, we conclude the following identity:

\begin{equation}\label{216.5}
\beta^2 = 4 (N_1+1) k_{1,1} \beta_1 +  \frac{4 (N_2+1)}{a} k_{1,2} \beta_2
\end{equation}

with

\begin{equation}\label{216.6}
\begin{split} 
&\beta_1 = \frac{1}{2\pi} \int_{\mathbb{R}^2} |x|^{2N_1} e^{v} = \frac{1}{2\pi} \int_{\mathbb{R}^2} |x|^{2N_1} e^{v_1},\\
&\beta_2 = \frac{1}{2\pi} \int_{\mathbb{R}^2} |x|^{2N_2} e^{a v} = \frac{1}{2\pi} \int_{\mathbb{R}^2} |x|^{2N_2} e^{v_2},\\
&\beta = k_{1,1} \beta_{1} + k_{1,2} \beta_{2}.
\end{split}
\end{equation}

At this point, by recalling that $a=\frac{k_{1,2}}{k_{2,2}}$ and $k_{1,1}k_{2,2}=k_{1,2}^2$, we easily derive \eqref{24} for $\beta_1$ and $\beta_2$ simply by inserting \eqref{216.6} into \eqref{216.5}.

Thus identity \eqref{24} is established in all cases, and the proof of the claim is completed.

To conclude the proof we shall see how to use \eqref{23bis} in order to reformulate the (non-symmetric) system \eqref{21} into a symmetric one, for which we can use \eqref{24}. To this purpose, we set,

\begin{equation*}
\hat{v}_1 (x) = v_1 (x)+\log|k_{2,1}| \quad \hat{v}_2 (x) = v_2 (x)+\log|k_{1,2}|
\end{equation*}

and observe that, if $(v_1,v_2)$ satisfies \eqref{21} then $(\hat{v}_1,\hat{v}_2)$ satisfies a similar problem with the \underline{symmetric} coupling matrix:

\begin{equation}\label{214}
K = \left( \begin{array}{cc} \frac{k_{11}}{|k_{21}|} & \pm 1\\
                                                           \pm 1 & \frac{k_{22}}{|k_{12}|}\\
																	\end{array} \right)
\end{equation}

and where the $\pm$ sign is chosen according to the sign of $k_{1,2}$ (or equivalently $k_{2,1}$). Furthermore,

\begin{equation}\label{215}
\hat{\beta}_1 := \frac{1}{2\pi} \int_{\mathbb{R}^2} |x|^{2N_1} e^{\hat{v}_1} = |k_{2,1}| \beta_1 \; \text{ and } \; \hat{\beta}_2 := \frac{1}{2\pi} \int_{\mathbb{R}^2} |x|^{2N_2} e^{\hat{v}_2}  = |k_{1,2}| \beta_2.
\end{equation}

Hence, we can apply the Claim to $(\hat{\beta}_1,\hat{\beta}_2).$ Thus, by virtue of \eqref{214}, from \eqref{24} we find:

\begin{equation}\label{216}
\frac{k_{1,1}}{|k_{2,1}|} \hat{\beta}_{1}^{2} + \frac{k_{2,2}}{|k_{1,2}|} \hat{\beta}_{2}^{2} \pm \hat{\beta}_{1} \hat{\beta}_{2} - 4 (N_1 + 1) \hat{\beta}_1 - 4 (N_2 + 1) \hat{\beta}_2 = 0,
\end{equation}

and so from \eqref{215} and  \eqref{216},  we readily derive \eqref{23}. 

\qed

\begin{remark}\label{rmk24}
It is easy to check that, in the symmetric case, i.e. $k_{1,2}= k_{2,1}$, the following identity holds:

\begin{equation}\label{216bis}
\begin{split} 
&(k_{1,1} k_{2,2} - k_{1,2}^2) [ (N_1+1) \beta_1 + (N_2+1) \beta_2)] =\\
&[(N_1+1)k_{2,2}-(N_2+1)k_{1,2}] \beta_{1}^{\infty} + [(N_2+1)k_{1,1}-(N_1+1)k_{1,2}] \beta_{2}^{\infty},
\end{split}
\end{equation}

and we can use it  together with  \eqref{212} to deduce the following equivalent formulation of \eqref{24}  expressed in terms of  $\beta_{i}^{\infty}$, $i=1,2$:

\begin{equation}\label{216*}
\begin{split}
&k_{2,2} (\beta_{1}^{\infty})^{2} + k_{1,1} (\beta_{2}^{\infty})^{2} - 2 k_{1,2} \beta_{1}^{\infty} \beta_{2}^{\infty} - 4[(N_1+1)k_{2,2}-(N_2+1)k_{1,2}] \beta_{1}^{\infty} \\ 
&- 4 [(N_2+1)k_{1,1}-(N_1+1)k_{1,2}] \beta_{2}^{\infty} = 0,
\end{split}
\end{equation}

which has the advantage to holds even for a \underline{degenerate} coupling matrix $K$.

Arguing as above, one can derive a similar identity for the non-symmetric case, provided that $k_{1,2}\cdot k_{2,1} > 0,$ we omit the details. 
\end{remark}

Clearly, Lemma \ref{lem1} and \ref{lem1a} imply Proposition \ref{pro12} and Corollary \ref{cor2a} as stated in the previous section.\\

Next, we turn to analyse the radial solvability of \eqref{21}. To this purpose, we observe that a \underline{radial} solution (about the origin) $(v_1(r),v_2(r))$  of \eqref{21}, may be characterised  equivalently as satisfying the following boundary value problem:

\begin{equation}\label{217}
\begin{cases}
-(rv'_{1} (r))'= k_{1,1} r^{2 N_1 +1} e^{v_1 (r)} + k_{1,2} r^{2 N_2 +1} e^{v_2 (r)}, \quad r>0 \\
-(rv'_{2} (r))'= k_{2,1} r^{2 N_1 +1} e^{v_1 (r)} + k_{2,2} r^{2 N_2 +1} e^{v_2 (r)}, \quad r>0 \\
v'_{i}(0)=0 , \quad \lim_{r \to \infty} r v'_{i}(r) = - \beta_{i}^{\infty},\\
v_i \in C([0,\infty)), \quad i=1,2,
\end{cases}
\end{equation}

with $\beta_{i}^{\infty}\quad i=1,2,$ defined in \eqref{129}.\\


To this purpose, we consider  the  Initial Value Problem associated to the system of ODE' s in \eqref{217}, and show that it admits a unique globally defined solutions which also accommodate  the boundary condition required in \eqref{217}. To be more precise we let,
\begin{equation}\label{219}
f_i (r) = \int_{0}^{r} s^{2 N_i + 1} e^{v_i (s)} \, ds
\end{equation}




and observe that we can express any \underline{local} solution of the system of ODE' s in \eqref{217}  defined in the interval $I= [0, R)$, equivalently as follows:

\begin{equation}\label{220}
\begin{cases}
rv'_{1} (r) + k_{1,1} f_1 (r) + k_{1,2} f_2 (r)=0, \quad 0 < r < R,  \\
rv'_{2} (r) + k_{2,1} f_1 (r) + k_{2,2} f_2 (r)=0, \quad 0 <r < R, \\
v'_{i}(0)=0 , v_i \in C([0 , R)), \quad i=1,2.
\end{cases}
\end{equation}

We have:

\begin{lemma}\label{lem23}
If the coupling matrix $K$ is \underline{symmetric} and $(v_1(r), v_2(r))$ satisfies \eqref{220} with $f_i = f_i (r)$ in \eqref{219}, then:

\begin{equation}\label{222}
\begin{split}
&i) \quad r^{2(N_1 +1)} e^{v_1 (r)} + r^{2(N_2 +1)} e^{v_2 (r)} - 2(N_1 +1) f_1(r) - 2(N_2 +1) f_2(r)+\\
&\frac{1}{2} (k_{1,1} f_1^2 (r) + 2 k_{1,2} f_1 (r) f_2 (r) + k_{2,2} f_2^2 (r))=0, \; \forall \, 0 < r < R;
\end{split}
\end{equation}

\begin{equation}\label{223}
\begin{split}
&ii) \quad r^{2(N_i +1)} e^{v_i (r)} - 2(N_i +1) f_i (r) + \frac{1}{2} k_{i,i} f_i ^2(r) = - \int_{0}^{r} k_{i,j} \dot{f}_i (s) f_j (s) \, ds,\\
&\; \forall \, 0 < r < R, \; \forall i \ne j \in \{ 1,2\}.
\end{split}
\end{equation}
\end{lemma}

\proof The above identities were pointed out in \cite{pot2} for more general systems, and are obtained simply by computing the derivative of the term in the left hand side of \eqref{222} and \eqref{223} respectively, and by using \eqref{220}.  

\qed







Again we observe that, as in the proof of Lemma \ref{lem1a}, it is possible to derive identities analogous to \eqref{222} and \eqref{223}, also in case we drop the symmetric assumption on the matrix $K$, but we assume instead that: $k_{1,2}\cdot k_{2,1} > 0$, we omit the details.\\

Furthermore, for solutions of \eqref{217}, we can deduce \eqref{124a} and \eqref{124b} as a direct consequence of \eqref{222} and \eqref{223}, just by letting $r\to \infty.$

\begin{remark}\label{rmk20}
In case $K$ is strictly positive definite, (a necessary condition when $k_{1,2} = k_{2,1} \leq 0$), from \eqref{222} we obtain that the monotone  functions $f_1(r) \text{ and } f_2(r)$ in \eqref{219} are uniformly bounded by a constant depending only on $N_1$ and $N_2.$  Therefore,  by a standard Picard iterative scheme, we see that any (local) solution of the Cauchy Problem associated to the system of ODE' s in \eqref{217} can be globally extended in $[0, \infty)$ to define a solution of \eqref{217}  with 
$\beta_i :=\lim_{r \to \infty} f_i (r), \text{ and } \beta_{i}^{\infty}$ defined by \eqref{129}, $i=1,2.$ 
\end{remark}



\section {Main Results and their Proof}
\setcounter{equation}{0}

We devote this section to analyse radial solutions for problem $(P)_\tau$ with $\tau \in (0,1)$, or equivalently to investigate solutions $(v_1(r), v_2(r))$ of the problem:

\begin{equation}\label{31}
\begin{cases}
-(rv'_{1} (r))'= r^{2 N +1} e^{v_1 (r)} -\tau r e^{v_2 (r)}, \quad r>0, \\
-(rv'_{2} (r))'= r e^{v_2 (r)} - \tau r^{2 N +1} e^{v_1 (r)}, \quad r>0, \\
v'_{i}(0)=0 , \quad \lim_{r \to \infty} r v'_{i}(r) = - \beta_{i,\tau},\\
v_i \in C([0,\infty)), \quad i=1,2,
\end{cases}
\end{equation}


with
\begin{equation}\label{32bis}
\beta_{1,\tau} := \beta_1 - \tau \beta_2 \quad \text{ and }  \quad \beta_{2,\tau} := \beta_2 - \tau \beta_1, 
\end{equation}
and  
\begin{equation}\label{32*}
\beta_1 := \int^{\infty}_{0}r^{2N+1}e^{v_1}dr \quad { and }\quad \beta_2 := \int^{\infty}_{0}re^{v_2}dr;
\end{equation}
or equivalently,
\begin{equation}\label{32} 
 \beta_1=\frac{\beta_{1,\tau} + \tau \beta_{2,\tau}}{1-\tau^2} \text{ and } \; \beta_2=\frac{\beta_{2,\tau} + \tau \beta_{1,\tau}}{1-\tau^2}.
\end{equation}

According to the results established in the previous section, we know that the following conditions are \underline {necessary} for the solvability of \eqref{31}:

\begin{eqnarray}
\label{33}&\beta_{1}^{2} + \beta_{2}^{2} - 2 \tau \beta_{1} \beta_{2} - 4(N+1) \beta_1 - 4 \beta_2 = 0,\\
\label{34}&\beta_1>4(N+1), \quad \beta_2 >4,\\
\label{35}&\beta_{1,\tau} = \beta_1 - \tau \beta_2 > 2(N+1), \quad \beta_{2,\tau} = \beta_2 - \tau \beta_1 > 2;
\end{eqnarray}

and we recall that  \eqref{33} can be expressed in terms of $\beta_{1,\tau}$ and $\beta_{2,\tau}$ equivalently as follows:
\begin{equation}\label{36}
\beta_{1,\tau}^{2} + \beta_{2,\tau}^{2} - 2 \tau \beta_{1,\tau} \beta_{2,\tau} = 4[(N+1+\tau) \beta_{1,\tau} +(1+\tau(N+1)) \beta_{2,\tau}],
\end{equation}
or as follows:
\begin{equation}\label{37}
\beta_1 (\beta_{1,\tau} - 2 (N+1)) + \beta_2 (\beta_{2,\tau}-2)= 2(N+1) \beta_{1} + 2\beta_{2}.
\end{equation}

Our main goal in this section is to investigate to what extent the conditions \eqref{33}, \eqref{34} and \eqref{35} are also sufficient for the solvability of \eqref{31}.\\
To this purpose we recall that problem \eqref{31} is invariant under the following scaling property:
\begin{equation}\label{38}
\begin{split}
v_1 (r) \to v_{1,\lambda} (r) = v_1 (\lambda r) + 2 (N+1) \log \lambda\\
v_2 (r) \to v_{2,\lambda} (r) = v_2 (\lambda r) + 2 \log \lambda,
\end{split}
\end{equation}

in the sense that: $(v_1, v_2)$ solves $(P)_\tau$ if and only if $(v_{1,\lambda}, v_{2,\lambda})$ solves $(P)_\tau$. We shall distinguish between this 1-parameter family of solutions by their initial value  at $r=0.$\\ 

Notice also that, for $\tau \in (0,1)$, the identity \eqref{222} for problem \eqref{31} reads as follows:
\begin{equation}\label{39}
\begin{split}
\left( \int_{0}^{r} t^{2N+1} e^{v_{1} (t)} \, dt\right)^2 + \left( \int_{0}^{r} t e^{v_{2} (t)} \, dt\right)^2 - 2 \tau \left( \int_{0}^{r} t^{2N+1} e^{v_{1} (t)} \, dt\right) \left( \int_{0}^{r} t e^{v_{2} (t)} \, dt\right)\\
- 4 (N+1) \left( \int_{0}^{r} t^{2N+1} e^{v_{1} (t)} \, dt\right) - 4 \left( \int_{0}^{r} t e^{v_{2} (t)} \, dt\right) + 2r^{2(N+1)} e^{v_{1} (r)} + 2 r^{2} e^{v_{2} (r)}=0,
\end{split}
\end{equation}
for all $r \geq 0.$ \\

From \eqref{39} we obtain the following uniform estimate about solution of \eqref{31}, independently from their initial data and the values of $\beta_1$ and $\beta_2:$
\begin{lemma}\label{lem32}
Let $\tau \in (0,1)$ and  $(v_1,v_2)$ be a solution of \eqref{31}. There exists a suitable constant  $C=C(\tau, N)>0,$ depending \underline{only} on $\tau$ and $N,$ such that:  
\begin{equation}\label{2114bis}
v_1(r)+2(N+1){\rm{log}}r\leq C\mbox{ and } \quad v_2(r)+2{\rm{log}}r\leq C,\quad\mbox{for any }r>0.
\end{equation}
\end{lemma}
\begin{proof}
 

By direct calculation, from \eqref{39} we easily derive the following estimate:
\begin{equation}\label{313}
r^{2(N+1)}e^{v_1 (r)} + r^{2}e^{v_2 (r)} \leq \frac{2[(N+1)^2+2\tau(N+1)+1]}{1-\tau^2},
\end{equation}
and by \eqref{313}, we readily derive \eqref{2114bis}. 

\qed

An important and useful consequence of \eqref{2114bis} is given by the following:

\begin{lemma}[Harnack's inequality] \label{harn}
Let $\tau \in (0,1)$ and $\left(v_1,v_2\right)$ be a solution of \eqref{31}. For every $0<r_0<R_0$, there exists a constant $C_0=C_0(\tau, N, \dfrac{R_0}{r_0})>0$ (depending only on $\tau$, $N$ and $\dfrac{R_0}{r_0}$) and $\gamma=\gamma\left(\dfrac{R_0}{r_0}\right)\in(0,1)$ such that,
\begin{equation}\label{Harn}
\begin{split}
&\max_{\left[r_0,R_0\right]}v_1\leq\gamma\min_{\left[r_0,R_0\right]}v_1+2(N+1)(\gamma-1)\log r_0+C_0,\\
&\max_{\left[r_0,R_0\right]}v_2\leq\gamma\min_{\left[r_0,R_0\right]}v_2+2(\gamma-1)\log r_0+C_0.
\end{split}
\end{equation}
\end{lemma}
\begin{proof}
Set $L_0:=\dfrac{R _0}{r_0}$ and define $D=\left\{\dfrac{1}{2}<r<2L_0\right\}$. We rescale $\left(v_1, v_2\right)$ as follows:

\begin{equation}\label{5217bis}
\hat{v}_1(r)=v_1(r_0 r)+2(N+1){\rm{log}}r_0\quad\mbox{ and }\quad \hat{v}_2(r)=v_2(r_0 r)+2{\rm{log}}r_0,
\end{equation}

so that, in view of \eqref{38}, the pair  $(\hat{v}_1,\hat{v}_2)$ still satisfies \eqref{31} together with \eqref{2114bis}.
In particular by setting:
$$
\hat{f}_1(r) = r^{2N+1} e^{\hat{v}_1 (r)} - \tau r e^{\hat{v}_2} \quad \text{ and }\quad  \hat{f}_2(r) = r e^{\hat{v}_2 (r)} - \tau r^{2N+1} e^{\hat{v}_1},
$$
by \eqref{2114bis} we find a constant $\hat{C}=\hat{C}(\tau, N)>0$, depending only on $\tau$ and $N$, such that:  $\hat{v}_{i}\leq \hat{C}$ in $D,$ and $\left\|\hat{f}_{i}\right\|_{L^{\infty}(D)}\leq \hat{C}$, $i=1,2.$

Therefore, for $i=1,2,$ we let $\psi_i$ the unique solution of the following Dirichlet problem:
\begin{equation}
\begin{cases}
-(r\psi'_i)'=\hat{{f}_i}\quad\mbox{ in }D\\
\psi_i=0\quad\mbox{ on }\quad \partial D\\
\end{cases}\quad i=1,2,
\end{equation}
and derive, by standard elliptic estimates (see \cite{gt}), that,
$$
\max_{\overline{D}}\left|\psi_i\right|\leq A \left\|\hat{f}_i\right\|_{L^{\infty}(D)}
$$
with a suitable constant $A=A(L_0)>0$.
As a consequence, or  $i=1,2,$ the function $\phi_i:=\psi_i-(\hat{v}_i-\hat{C})$ defines a positive harmonic function in $D,$ for which we can use Harnack inequality in $D':=\left\{1<r<L_0\right\}\subset\subset D$, and obtain:
$$
\sup_{D'}\phi_i\leq\dfrac{1}{\gamma}\inf_{D'}\phi_i,\quad i=1,2.
$$
with a (universal) constant $\gamma=\gamma(L_0)\in (0,1).$
In other words,
\begin{equation*}
\sup_{D'}\hat{v}_i\leq\gamma\inf_{D'}\hat{v}_i+(1-\gamma)\hat{C}+(1+\gamma)\max\left|\psi_i\right|
\end{equation*}
with $\gamma=\gamma(L_0)\in (0,1)$, and the desired estimate \eqref{Harn} follows.

\qed

Finally, we observe that, in view or Remark \ref{rmk20}, for $\tau\in(0,1)$ and for any $(\alpha_1,\alpha_2)\in\mathbb{R}^{2}$, the initial value problem:
\begin{equation}\label{39tris}
\begin{cases}
-(rv'_{1})'=r^{2N+1}e^{v_1}-\tau re^{v_2}\\
-(rv'_{2})'=re^{v_2}-\tau r^{2N+1}e^{v_1}\\
v_{i}(0)=\alpha_j,\quad v'_{i}(0)=0\qquad i=1,2
\end{cases}
\end{equation}
admits a \underline{unique} solution defined for all $r\geq0$, and such that:
$$
\int^{\infty}_{0}r^{2N+1}e^{v_1}dr < C\mbox{ and }\int^{\infty}_{0}re^{v_2}dr < C,
$$
with a suitable constant $C=C(N,\tau)>0$ depending only on $N$ and $\tau,$ (but independent of the initial data $(\alpha_1,\alpha_2)$).

\begin{remark}\label{rmk22}
By virtue of the scale invariant property \eqref{38}, it follows that the set of pairs $(\beta_1,\beta_2)$ for which \eqref{31}, \eqref{32bis} admits a solution is fully described by the curve:
\begin{equation}\label{310}
(\beta_1(\alpha),\beta_2(\alpha)), \quad \alpha \in \mathbb{R}
\end{equation}
where, 
\begin{equation}\label{311}
\beta_1(\alpha) = \int^{\infty}_{0} r^{2N+1}e^{v_1 (r,\alpha)}\, ds, \qquad \beta_2(\alpha) = \int^{\infty}_{0} r e^{v_2 (r,\alpha)}\, ds,
\end{equation}
and $(v_1 (r,\alpha),v_1 (r,\alpha))$ is the \underline{unique} solution of the initial value problem \eqref{39tris} with initial data specified as follows:
\begin{equation}\label{312}
\alpha_1=\alpha \, \text{ and } \, \alpha_2 = 0.
\end{equation}
\end{remark}

Our main purpose will be to identify the limiting values of $\beta_i (\alpha)$, $i=1,2$, as $\alpha \to \pm \infty$. \\
To this end, we start to give a detailed description about the portion $\mathcal{E}$ of the ellipse defined by \eqref{222} which is contained in the first quadrant of the $(\beta_1, \beta_2)$-plane, namely:
\begin{equation}\label{5234bis}
\mathcal{E}:
\begin{cases}
\beta^{2}_{1}+\beta^{2}_{2}-2\tau\beta_1\beta_2-4(N+1)\beta_1-4\beta_2=0\\
\beta_1>0, \beta_2>0.
\end{cases}
\end{equation}
For this purpose, it is useful to identify  first the intersection of $\mathcal {E}$ with the line: $\beta_1-\tau\beta_2=2(N+1)$ and the line:  $\beta_1-\tau\beta_1=2$. To this end, we set:
\begin{equation}\label{alldinuo}
(N+1)^{2}+2\tau(N+1)+1=:D(\tau,N) 
\end{equation}
and by direct calculations, we verify the following:

\begin{equation}\label{ses1}
\begin{cases}
\beta_1>0,\beta_2>0\\
\beta^{2}_{1}+\beta^{2}_{2}-2\tau\beta_{1}\beta_{2}-4(N+1)\beta_1-4\beta_2=0\\
\beta_1-\tau\beta_2=2(N+1)
\end{cases}\Longleftrightarrow
\begin{cases}
\begin{aligned}
\beta_1&=\dfrac{2}{1-\tau^2}\left(N+1+\tau+\tau\sqrt{D}\right)=:\underline{\beta}_{1}(\tau)\\
\beta_2&=\dfrac{2}{1-\tau^2}\left(1+\tau(N+1)+\sqrt{D}\right)=:\overline{\beta}_{2}(\tau);
\end{aligned}
\end{cases}
\end{equation}

and similarly:

\begin{equation}\label{ses2}
\begin{cases}
\beta_1>0,\beta_2>0\\
\beta^{2}_{1}+\beta^{2}_{2}-2\tau\beta_{1}\beta_{2}-4(N+1)\beta_1-4\beta_2=0\\
\beta_2-\tau\beta_1=2
\end{cases}\Longleftrightarrow
\begin{cases}
\begin{aligned}
\beta_1&=\dfrac{2}{1-\tau^2}\left(N+1+\tau+\sqrt{D}\right)=:\overline{\beta}_{1}(\tau)\\
\beta_2&=\dfrac{2}{1-\tau^2}\left(1+\tau(N+1)+\tau\sqrt{D}\right)=:\underline{\beta}_{2}(\tau).
\end{aligned}
\end{cases}
\end{equation}

Furthermore, we can explicitly solve the quadratic equation in \eqref{222} in terms of $\beta_1$ or $\beta_2$. In this way, we are lead to consider the functions:

\begin{equation}\label{funct1}
\varphi^{\pm}_{1}(\beta_1)=2+\tau\beta_1\pm\sqrt{(2+\tau\beta_1)^{2}-\beta_{1}\left(\beta_1-4(N+1)\right)},\quad\mbox{with }0<\beta_1\leq\overline{\beta}_{1}(\tau)
\end{equation}

and

\begin{equation}\label{funct2}
\varphi^{\pm}_{2}(\beta_2)=2(N+1)+\tau\beta_2\pm\sqrt{(2(N+1)+\tau\beta_2)^{2}-\beta_{2}\left(\beta_2-4\right)},\quad\mbox{with }0<\beta_2\leq\overline{\beta}_{2}(\tau).
\end{equation}

To simplify notations, from now on we shall drop the dependence of $\underline{\beta}_{i}(\tau)$ and $\overline{\beta}_{i}(\tau)$ $i=1,2,$ on the parameter $\tau.$ 

By direct inspection, one can easily check the following:

\begin{lemma}\label{lemmaalln}
We have:
\begin{itemize}
\item[$(i)$] $\varphi^{+}_{1}:\left[\underline{\beta}_{1}, \overline{\beta}_{1}\right]\to\left[\underline{\beta}_{2}, \overline{\beta}_{2}\right]$ is strictly monotone decreasing with inverse 

$$(\varphi^{+}_{1})^{-1}=\varphi^{+}_{2}:[\underline{\beta}_2,\overline{\beta}_2]\to[\underline{\beta}_1,\overline{\beta}_1];$$
\item[$(ii)$]$\varphi^{-}_{1}:\left[4(N+1), \overline{\beta}_{1}\right]\to\left[0, \underline{\beta}_{2}\right]$ is strictly monotone increasing;
\item[$(iii)$]$\varphi^{-}_{2}:\left[4, \overline{\beta}_{2}\right]\to\left[0, \underline{\beta}_{1}\right]$ is strictly monotone increasing.
\end{itemize}
\end{lemma}
\qed

For later use, we emphasise that,

\begin{equation}\label{all114}
\varphi^{\pm}_{1}(\overline{\beta}_{1})=\underline{\beta}_{2},
\end{equation}
\begin{equation}\label{all114bis}
\varphi^{\pm}_{2}(\overline{\beta}_{2})=\underline{\beta}_{1},
\end{equation}

and moreover, for $(\beta_1,\beta_2)\in\mathcal{E}$ there holds:

\begin{equation}\label{all113}
0<\beta_1\leq\overline{\beta}_{1}\quad\mbox{ and }\quad 0<\beta_2\leq\overline{\beta_{2}}.
\end{equation}

Next, by recalling \eqref{37}, we observe that $(\beta_{1,\tau}-2(N+1))$ and $(\beta_{2,\tau}-2)$ cannot be simultaneously negative. Thus we obtain the following description of $\mathcal{E}$:

\begin{lemma}\label{lemmadueall}
The pair $\left(\beta_1,\beta_2\right)\in\mathcal{E}$ if and only if it satisfies one of the following set of conditions:
\begin{equation}\label{all110}
\begin{split}
(i) \quad &\beta_{1,\tau}>2(N+1) \;\text { and } \; \beta_{2,\tau}>2, \quad \beta_1\in\left(\underline{\beta}_{1},\overline{\beta}_{1}\right)\quad\mbox{ and }\quad \beta_2=\varphi^{+}_{1}(\beta_1)\\
&(\mbox{or equivalently,} \quad \beta_2\in\left(\underline{\beta}_{2},\overline{\beta}_{2}\right)\quad\mbox{ and }\quad \beta_1=\varphi^{+}_{2}(\beta_2));
\end{split}
\end{equation}

\begin{equation}\label{all111}
(ii) \quad \beta_{1,\tau}\leq 2(N+1) \; \text{ and }\; \beta_{2,\tau}>2, \, \beta_2\in(4,\overline{\beta}_{2}]\,\mbox{ and }\,\beta_1=\varphi^{-}_{2}(\beta_2)\in(0,\underline{\beta_1}];
\end{equation}

\begin{equation}\label{all112}
(iii) \quad \beta_{1,\tau}> 2(N+1) \; \text{ and  } \; \beta_{2,\tau}\leq 2, \; \beta_1\in(4(N+1),\overline{\beta}_{1}]\,\mbox{ and }\, \beta_2=\varphi^{-}_{1}(\beta_1)\in(0,\underline{\beta_2}].
\end{equation}

\end{lemma}

\begin{proof}
From \eqref{37} we see that
\eqref{all110}-\eqref{all112} cover all the possibilities, and the relative expression for $(\beta_1,\beta_2)$ can be easily verified by virtue of the definition of $\varphi^{\pm}_{1}$ and $\varphi^{\pm}_{2}$.

\end{proof}

Observe that,  \eqref{all110} accounts for the conditions \eqref{33} and \eqref{35}. In order to account also for \eqref{34}, we search for the intersections of the set $\mathcal{E}$ in \eqref{5234bis} with the lines $\beta_1=4(N+1)$ and $\beta_2=4$ respectively, and find:

\begin{equation}\label{all115}
\begin{cases}
\beta^{2}_{1}+\beta^{2}_{2}-2\tau\beta_{1}\beta_{2}-4(N+1)\beta_1-4\beta_2=0\\
\beta_1=4(N+1) \text{ and } \beta_2>0
\end{cases}\Longleftrightarrow
\beta_2=4+8\tau(N+1)=:\beta^{*}_{2}(\tau)
\end{equation}

and

\begin{equation}\label{all116}
\begin{cases}
\beta^{2}_{1}+\beta^{2}_{2}-2\tau\beta_{1}\beta_{2}-4(N+1)\beta_1-4\beta_2=0\\
\beta_2=4 \text{ and } \beta_2>0
\end{cases}\Longleftrightarrow
\beta_1=4(N+1)+8\tau=:\beta^{*}_{1}(\tau).
\end{equation}

Also, we define:
\begin{equation}\label{all117}
\beta^{**}_{1}(\tau):=2\tau\beta^{*}_{2}(\tau)=8\tau\left(1+2\tau(N+1)\right), \quad \beta^{**}_{2}(\tau):=2\tau\beta^{*}_{1}(\tau)=8\tau\left(N+1+2\tau\right),
\end{equation}

and observe that, not only $(4(N+1),\beta^{*}_{2}(\tau)) \text{ and } (\beta^{*}_{1}(\tau),4)\in\mathcal{E},$ but also:

\begin{equation}\label{all118}
\left(\beta^{**}_{1}(\tau),\beta^{*}_{2}(\tau)\right) \text{ and } \left(\beta^{*}_{1}(\tau),\beta^{**}_{2}(\tau)\right)\in\mathcal{E}.
\end{equation}

We expect such values to play a role in identifying the pairs $\left(\beta_1,\beta_2\right)$ for which  \eqref{31}-\eqref{32*} admits a solution. 

\begin{prop}\label{527gir}
We have:
\begin{itemize}
\item[$(i)$] \begin{equation}\label{all119}
\begin{aligned}
\beta^{**}_{1}(\tau)=4(N+1)&\Longleftrightarrow\underline{\beta}_{1}(\tau)=4(N+1) \text{ and } \beta^{*}_{2}(\tau)=\overline{\beta}_{2}(\tau)\\
&\Longleftrightarrow\tau=\dfrac{N+1}{1+\sqrt{1+4(N+1)^{2}}}=:\tau^{(1)}_0.
\end{aligned} 
\end{equation}
In particular, 
\begin{equation}\label{52511}
\beta^{**}_{1}(\tau)-\tau\beta^{*}_{2}(\tau)=2(N+1)\Longleftrightarrow\tau=\tau^{(1)}_{0}.
\end{equation}
\item[$(ii)$]\begin{equation}\label{all120}
\begin{aligned}
\beta^{**}_{2}(\tau)=&4\Longleftrightarrow\underline{\beta}_{2}(\tau)=4,\; \text{ and } \beta^{*}_{1}(\tau)=\overline{\beta}_{1}(\tau)\\
&\Longleftrightarrow \tau=\dfrac{1}{N+1+\sqrt{(N+1)^{2}+4}}=:\tau^{(2)}_0.
\end{aligned}
\end{equation}
In particular, 
\begin{equation}\label{5252}
\beta^{**}_{2}(\tau)-\tau\beta^{*}_{1}(\tau)=2\Longleftrightarrow\tau=\tau^{(2)}_{0}.
\end{equation}
\end{itemize}
Moreover:
\begin{equation}\label{259bis}
0<\tau^{(2)}_{0}<\tau^{(1)}_{0}<\frac{1}{2}.
\end{equation}
\end{prop}

\begin{proof}
By straightforward calculations we see that, $\beta_1^{**}(\tau)=4(N+1)$ if and only if $\tau=\tau^{(1)}_{0}$. Furthermore recalling that: $\beta^{**}_{1}(\tau)=2\tau\beta^{*}_{2}(\tau),$ we also find that $4(N+1)-\tau\beta^{*}_{2}(\tau)=2(N+1)$ if and only if $\tau=\tau^{(1)}_{0},$ and therefore: $\underline{\beta}_{1}(\tau)=4(N+1)$ and $\beta^{*}_{2}(\tau)=\overline{\beta}_{2}(\tau)$ if and only if $\tau=\tau^{(1)}_{0}$.

So $(i)$ is established, and $(ii)$ follows exactly in the same way.

Finally to check \eqref{259bis}, we observe that the function $f(t)=\dfrac{1}{t+\sqrt{{t}^{2}+4}}$, is strictly decreasing for $t>0$. Hence, for $N>0$, we have:
$$
\tau^{(2)}_{0}=f(N+1)<f\left(\dfrac{1}{N+1}\right)=\tau^{(1)}_{0}.
$$

\end{proof}

As an immediate consequence of Proposition~\ref{527gir} we obtain:
\begin{coro}\label{coro41}
For every $\tau\in(0,1)$ we have:
\begin{itemize}
\item[$(i)$] $0<\beta^{*}_{1}(\tau)\leq\overline{\beta}_{1}(\tau)$ and equality holds if and only if $\tau=\tau^{(2)}_{0}$.
\item[$(ii)$] $0<\beta^{*}_{2}(\tau)\leq\overline {\beta}_{2}(\tau)$ and equality holds if and only if $\tau=\tau^{(1)}_{0}$.
\end{itemize}
\end{coro}

\begin{proof}
In view of \eqref{all113} and \eqref{all114}, we see that properties $(i)$ and $(ii)$ follow easily from Proposition~\ref{527gir} .

\end{proof}

\begin{coro}\label{42} There holds:
\begin{itemize}
\item[$(i)$] \begin{equation}\label{coro5281} \begin{aligned}
\underline{\beta}_{1}(\tau)<4(N+1)&\Longleftrightarrow 4(N+1)-\tau\beta^{*}_{2}(\tau)>2(N+1)>\beta^{**}_{1}(\tau)-\tau\beta^{*}_{2}(\tau)\\
&\Longleftrightarrow\beta^{**}_{1}(\tau)<4(N+1)\Longleftrightarrow\tau\in\left(0,\tau^{(1)}_{0}\right).
\end{aligned}
\end{equation}
Moreover, for every $\tau\in\left(0,\tau^{(1)}_{0}\right)$ we have:
\begin{equation}\label{coro5282}
\begin{aligned}
&\underline{\beta}_{1}(\tau)<4(N+1)=\varphi^{+}_{2}(\beta^{*}_{2}(\tau))\quad(\mbox{or equivalently }\quad \beta^{*}_{2}(\tau)=\varphi^{+}_{1}(4(N+1))<\overline{\beta}_{2}(\tau))\\
&\mbox{ and }\beta^{**}_{1}(\tau)=\varphi^{-}_{2}(\beta^{*}_{2}(\tau))<\underline{\beta}_{1}(\tau).
\end{aligned}
\end{equation}
\item[$(ii)$] \begin{equation} \begin{aligned}
\underline{\beta}_{2}(\tau)<4&\Longleftrightarrow 4-\tau\beta^{*}_{1}(\tau)>2>\beta^{**}_{2}(\tau)-\tau\beta^{*}_{1}(\tau)\Longleftrightarrow\beta^{**}_{2}(\tau)<4\\
&\Longleftrightarrow\tau\in\left(0,\tau^{(2)}_{0}\right).
\end{aligned}
\end{equation}
\end{itemize}
Moreover, for every $\tau\in\left(0,\tau^{(2)}_{0}\right)$ we have:
\begin{equation}
\begin{aligned}
&\underline{\beta}_{2}(\tau)<4=\varphi^{+}_{1}(\beta^{*}_{1}(\tau))\quad(\mbox{or equivalently }\quad \beta^{*}_{1}(\tau)=\varphi^{+}_{2}(4)<\overline{\beta}_{1}(\tau))\\
&\mbox{ and }\beta^{**}_{2}(\tau)=\varphi^{-}_{1}(\beta^{*}_{1}(\tau))<\underline{\beta}_{2}(\tau).
\end{aligned}
\end{equation}
\end{coro}

\begin{proof}
We prove $(i)$ since $(ii)$ follows in a similar way. The properties in \eqref{coro5281} follow by Proposition~\ref{527gir} and the monotonicity of the functions involved.
To show \eqref{coro5282}, we observe first that, for $\beta_1=4(N+1)$ and $\beta_2=\beta^{*}_{2}(\tau)$ we know that,
$\beta_1-\tau\beta_2>2(N+1)$ and we can also check that, $\beta_2-\tau\beta_1>2$, $\forall \tau\in\left(0,\tau^{(1)}_{0}\right)$.
So, we can use Lemma~\ref{lemmadueall} together with the monotonicity of the function $\varphi^{+}_{1}$ and $\varphi^{+}_{2}$ given in Lemma~\ref{lemmaalln}, to conclude that, $\underline{\beta}_{1}(\tau)=\varphi^{+}_{2}(\overline{\beta}_{2}(\tau))<\varphi^{+}_{2}(\beta^{*}_{2})=4(N+1)$, as claimed. The dual statement in \eqref{coro5282} follows similarly. Furthermore, since for $\beta_1=\beta^{**}_{1}$ and $\beta_2=\beta^{*}_{2}$ we know that, $\beta_1-\tau\beta_2<2(N+1)$, $\forall\tau\in\left(0,\tau^{(1)}_{0}\right)$, and at the same time we can check that, $\beta_2-\tau\beta_1> 2$, so we can use again Lemma~\ref{lemmadueall} to derive: $\beta^{**}_{1}=\varphi^{-}_{2}(\beta^{*}_{2})< \underline{\beta}_{1},$ as claimed.

\end{proof}

\begin{remark}\label{rmk23}
From the above discussion, we see that for $\tau\in(0,\tau^{(2)}_{0}]$ the (necessary) conditions \eqref{33}, \eqref{34} imply the integrability condition \eqref{35}, as it happens for the cooperative case $\tau\leq 0$.  Thus, we expect that, when
$\tau\in(0,\tau^{(2)}_{0}],$ then the same uniqueness and non-degeneracy properties (as established in Theorem \ref{teo2} for $\tau < 0$) should remain valid for radial solutions of $(P_{\tau})$. 
\end{remark}

Next, we wish to describe the relations between $\beta^{**}_{1}(\tau)$ and $\beta^{*}_{2}(\tau)$ when $\tau\in\left(\tau^{(1)}_{0},1\right)$, and $\beta^{**}_{2}(\tau)$ and $\beta^{*}_{1}(\tau)$ when $\tau\in\left(\tau^{(2)}_{0},1\right)$.
Since for $\tau\in\left(\tau^{(1)}_{0},1\right)$ we can always guarantee that,
\begin{equation}\label{equat15}
\beta^{**}_{1}(\tau)-\tau\beta^{*}_{2}(\tau)>2(N+1),
\end{equation}
thus, we need to investigate when we can also realise the second integrability condition:
\begin{equation}\label{equat16}
\beta^{*}_{2}(\tau)-\tau\beta^{**}_{1}(\tau)>2.
\end{equation}
Similar considerations can be applied about the relations between $\beta^{*}_{1}(\tau)$ and $\beta^{**}_{2}(\tau)$.

\begin{prop}\label{prop41}
\begin{itemize}
\item[$(i)$] There exists a {\bf{unique}} value $\tau^{(1)}_{1}\in(\frac{1}{2},\frac{1}{\sqrt{2}})$ such that, \\ $\beta^{*}_{2}(\tau)-\tau\beta^{**}_{1}(\tau)=2$, that is: $\beta^{*}_{2}(\tau)=\underline{\beta}_{2}(\tau)$ and $\beta^{**}_{1}(\tau)=\overline{\beta}_{1}(\tau)\Longleftrightarrow \tau=\tau^{(1)}_{1}$.

More precisely,
$$
\beta^{*}_{2}(\tau)-\tau\beta^{**}_{1}(\tau)>2\Longleftrightarrow \tau\in (0,\tau^{(1)}_{1})
$$
and
$$
0<\beta^{**}_{1}(\tau)<\overline{\beta}_{1}(\tau)\quad \forall \tau\in(0,1)\setminus\left\{\tau^{(1)}_{1}\right\}.
$$
\item[$(ii)$]
There exists a {\bf{unique}} value $\tau^{(2)}_{1}\in(\frac{1}{2},\frac{1}{\sqrt{2}})$ such that $\beta^{*}_{1}(\tau)-\tau\beta^{**}_{2}(\tau)=2(N+1),$ that is: $\beta^{*}_{1}(\tau)=\underline{\beta}_{1}(\tau)$ and $\beta^{**}_{2}(\tau)=\overline{\beta}_{2}(\tau)\Longleftrightarrow \tau=\tau^{(2)}_{1}$.

More precisely,
$$
\beta^{*}_{1}(\tau)-\tau\beta^{**}_{2}(\tau)>2(N+1)\Longleftrightarrow \tau\in (0,\tau^{(2)}_{1})
$$
and
$$
0<\beta^{**}_{2}(\tau)<\overline{\beta}_{2}(\tau)\quad \forall \tau\in(0,1)\setminus\left\{\tau^{(2)}_{1}\right\}.
$$
\end{itemize}
Moreover,
\begin{equation*}
\frac{1}{2} <\tau^{(2)}_{1}<\tau^{(1)}_{1}<\frac{1}{\sqrt{2}}.
\end{equation*}
\end{prop}

\begin{proof}
We establish $(i)$ with the help of the function, 

$$\psi_{1}(\tau)=\frac{1}{2}\left(\beta^{*}_{2}(\tau)-\tau\beta^{**}_{1}(\tau)\right)-1=2(1-2\tau^{2})(1+2\tau(N+1))-1.$$

We readily check that, for $\overline{\tau}=\dfrac{\sqrt{1+6(N+1)^{2}}-1}{6(N+1)}\in(0,\frac{1}{2})$, the function $\psi_1$ is increasing in $(0,\overline{\tau}]$ and decreasing in $(\overline{\tau},+\infty)$. Since $\psi_1(\frac{1}{2})=N+1>0$ while $\psi(\frac{1}{\sqrt{2}})=-1,$ we find a {\underline{unique}} $\tau^{(1)}_{1}\in\left(\frac{1}{2},\frac{1}{\sqrt{2}}\right):$ $\psi_{1}(\tau^{(1)}_{1})=0$, and moreover $\psi_{1}(\tau)>0$, $\forall \tau\in(0,\tau^{(1)}_{1})$. 

Finally from \eqref{all113} we know that, $0<\beta^{**}_{1}(\tau)\leq\overline{\beta}_{1}(\tau)$ and $(i)$ is established.

We observe that $(ii)$ follows in a similar way by considering the function 

$$\psi_{2}(\tau)=\dfrac{1}{2}(\beta^{*}_{1}(\tau)-\tau\beta^{**}_{2}(\tau))-2(N+1)=2(1-2\tau^2)(N+1+2\tau)-(N+1),$$ 

we omit the details.

Thus it remains to show that: $\tau^{(2)}_{1}<\tau^{(1)}_{1}$. To this purpose, we easily check that, $\forall\tau\in\left(\dfrac{1}{2},\dfrac{1}{\sqrt{2}}\right)$ there holds: $\psi_{2}(\tau)<\psi_{1}(\tau).$ 

So the desired conclusion follows by recalling that $\tau^{(1)}_{1}$ is the unique positive zero for $\psi_{1}(\tau)$, while $\tau^{(2)}_{1}$ is the unique positive zero for the function: $\psi_{2}(\tau),$ and $\tau^{(i)}_{1}\in\left(\dfrac{1}{2},\dfrac{1}{\sqrt{2}}\right)$, $i=1,2.$

\end{proof}

\begin{coro}\label{411}
\begin{itemize}
\item[$(i)$] If $\tau\in\left(\tau^{(1)}_{0}, \tau^{(1)}_{1}\right)$, then
\begin{equation}\label{corocoro110}
\beta^{**}_{1}(\tau)=\varphi^{+}_{2}\left(\beta^{*}_{2}(\tau)\right)\quad\mbox{ or equivalently }\quad \beta^{*}_{2}=\varphi^{+}_{1}\left(\beta^{**}_{1}(\tau)\right).
\end{equation}
In particular,
\begin{equation}\label{corocoro111}
4(N+1)<\underline{\beta}_{1}(\tau)<\beta^{**}_{1}(\tau)<\overline{\beta}_{1}(\tau)\quad\mbox{ and }\quad \underline{\beta}_{2}(\tau)<\beta^{*}_{2}(\tau)<\overline{\beta}_{2}(\tau).
\end{equation}
\item[$(ii)$] If $\tau\in\left(\tau^{(2)}_{0},\tau^{(2)}_{1}\right)$ then 
\begin{equation}\label{corocoro112}
\beta^{**}_{2}=\varphi^{+}_{1}(\beta^{*}_{1})\quad\mbox{ or equivalently }\quad \beta^{*}_{1}=\varphi^{+}_{2}(\beta^{**}_{2}).
\end{equation}
In particular, 
\begin{equation}\label{corocoro113}
4<\underline{\beta}_{2}(\tau)<\beta^{**}_{2}(\tau)<\overline{\beta}_{2}(\tau)\quad\mbox{ and }\quad \underline{\beta}_{1}(\tau)<\beta^{*}_{1}(\tau)<\overline{\beta}_{1}(\tau).
\end{equation}
\end{itemize}
\end{coro}

\begin{proof}
If $\tau\in\left(\tau^{(1)}_{0}, \tau^{(1)}_{1}\right)$ then we see that \eqref{equat15}, \eqref{equat16} hold simultaneously, and so we deduce \eqref{corocoro110} as a consequence of \eqref{all110}. Furthermore, as $\tau >\tau^{(1)}_{0}$ then $\beta^{*}_{2}(\tau)<\overline{\beta}_{2}(\tau)$, and so: $\beta^{**}_{1}(\tau)=\varphi^{+}_{2}(\beta^{*}_{2}(\tau))>\varphi^{+}_{2}(\overline{\beta}_{2}(\tau))=\underline{\beta}_{1}(\tau)$. Similarly, as $\tau < \tau^{(1)}_{1}$ then $\beta^{**}_{1}(\tau)<\overline{\beta}_{1}(\tau)$ and so $\beta^{*}_{2}(\tau)=\varphi^{+}_{1}(\beta^{**}_{1}(\tau))>\varphi^{+}_{1}(\overline{\beta}_{1}(\tau))=\underline{\beta}_{2}(\tau)$, and also \eqref{corocoro111} is established.
Part $(ii)$ follows exactly in the same way, with the obvious modifications, we omit the details.  

\end{proof}

Before we discuss our existence results for  problem $(P)_\tau$ we make the following simple observation:
\begin{equation}\label{611}
\beta^{**}_{i}(\tau) < \beta^{*}_{i}(\tau) \Longleftrightarrow \tau \in (0, \frac{1}{2}) \quad i=1,2
\end{equation}
and in particular,
\begin{equation}\label{612}
\beta^{**}_{j}(\tau) = \beta^{*}_{j}(\tau) = 4(N+2) \Longleftrightarrow \tau = \frac{1}{2}.
\end{equation}

In other words, when $\tau = \frac{1}{2}$ and problem $(P)_{\tau=\frac{1}{2}}$ reduces to the 2 X 2 -Toda-system, then for $i=1,2,$ the values $\beta^{**}_{i}$ and $\beta^{*}_{i}$, coincide with the  \underline{only} value allowed by solvablility, see \eqref{147*}.\\
Indeed, we show below that actually, the values $\beta^{**}_{i}$ and $\beta^{*}_{i}$ $i=1,2,$ capture in a crucial way the (radial) solvability for $(P)_\tau$ for any $\tau \in (0,1)$.

To this purpose, and in account of Proposition \ref{527gir}, Proposition \ref{prop41} and \eqref{611}, for $\tau \in (0,1)$, we define:

\begin{equation}\label{613}
\beta_{1}^{-} (\tau) = 
\begin{cases}
4 (N+1), \quad  0 < \tau \leq \tau^{(1)}_{0}\\
\beta^{**}_{1}(\tau), \quad  \tau^{(1)}_{0} < \tau \leq \frac12\\
\beta^{*}_{1}(\tau), \quad  \frac12 < \tau < \tau^{(2)}_{1}\\
\underline{\beta}_{1}(\tau), \quad  \tau^{(2)}_{1} \leq \tau < 1
\end{cases}
 \qquad
\beta_{1}^{+} (\tau) = 
\begin{cases}
\beta^{*}_{1}(\tau), \quad  0 < \tau \leq \frac12\\
\beta^{**}_{1}(\tau), \quad  \frac12 < \tau < \tau^{(1)}_{1}\\
\overline{\beta}_{1}(\tau), \quad  \tau^{(1)}_{1} \leq \tau < 1
\end{cases}
\end{equation}

and similarly,

\begin{equation}\label{613*}
\beta_{2}^{-} (\tau) = 
\begin{cases}
4, \quad  0 < \tau \leq \tau^{(2)}_{0}\\
\beta^{**}_{2} (\tau), \quad  \tau^{(2)}_{0} < \tau \leq \frac12\\
\beta^{*}_{2} (\tau), \quad  \frac12 < \tau < \tau^{(1)}_{1}\\
\underline{\beta}_{2}(\tau), \quad  \tau^{(1)}_{1} \leq \tau < 1
\end{cases}
 \qquad
\beta_{2}^{+} (\tau)= 
\begin{cases}
\beta^{*}_{2} (\tau), \quad  0 < \tau \leq \frac12\\
\beta^{**}_{2} (\tau), \quad  \frac12 < \tau < \tau^{(2)}_{1}\\
\overline{\beta}_{2}(\tau), \quad  \tau^{(2)}_{1} \leq \tau < 1
\end{cases}
\end{equation}

We see that, $\beta_{i}^{\pm}(\tau)$, $i=1,2$ is a continuous functions of $\tau$, and for $\tau \ne \frac{1}{2}$ we have:

\begin{eqnarray}
\label{615} &\max\{ \underline{\beta}_{1}(\tau), 4(N+1)\} \leq \beta_{1}^{-} (\tau)<\beta_{1}^{+} (\tau)\leq\overline{\beta}_{1}(\tau)\\
\label{616} &\max\{ \underline{\beta}_{2}(\tau), 4 \} \leq \beta_{2}^{-} (\tau)<\beta_{2}^{+} (\tau)\leq\overline{\beta}_{2}(\tau),
\end{eqnarray}
while, 
\begin{equation}\label{617}
\text{ if } \tau = \frac{1}{2} \quad \text{ then } \beta^{-}_{i}\left(\frac{1}{2}\right) = \beta^{+}_{i}\left(\frac{1}{2}\right) = 4(N+2), \quad i=1,2.
\end{equation}

We prove the following:

\begin{theo}\label{teoA}
If $\tau\in (0,1)$ and $\tau \ne \frac{1}{2}$, problem $(P)_\tau$ admits a \underline{radial} solution if and only if the pair $(\beta_1, \beta_2)$ satisfies the following:

\begin{equation}\label{618}
\beta_1 \in ( \beta_{1}^{-} (\tau),\beta_{1}^{+} (\tau) ) \quad \text{ and } \quad \beta_{2} = \varphi^{+}_{1} \left( \beta_{1} \right)
\end{equation}

or equivalently,

\begin{equation}\label{619}
\beta_2 \in (\beta_{2}^{-} (\tau),\beta_{2}^{+} (\tau))\quad \text{ and } \quad \beta_{1} = \varphi^{+}_{2}\left(\beta_{2}\right),
\end{equation}

with $\beta_i^{\pm}$ defined in \eqref{613} and \eqref{613*}, $i = 1, 2.$

\end{theo}

The proof of Theorem \ref{teoA} will require several preliminary steps, but first let us point out some of its interesting consequences. Indeed, in view of the definition of $\beta_{j}^{\pm} (\tau)$, $j=1,2$ in \eqref{613} and \eqref{613*}, we deduce 
the following:

\begin{coro}\label{coroA.1}
If $\tau \in (0, \tau^{(2)}_{0}]$ then the conditions \eqref{33}, \eqref{34} imply the integrability condition \eqref{35}, and they  are \underline{necessary} and \underline{sufficient} for the radial solvability of $(P)_\tau$.
\end{coro}

\proof
Simply observe that in this case: $\beta_{1}^{-} (\tau) = 4(N+1)$ and $\beta_{2}^{-} (\tau) = 4$, so the pairs $ (\beta_1, \beta_2)$ satisfying: $\beta_1 \in (\beta_{1}^{-} (\tau),\beta_{1}^{+} (\tau)),\beta_{2} = \varphi^{+}_{1}\left(\beta_{1}\right)$ (or equivalently $\beta_2 \in (\beta_{2}^{-} (\tau),\beta_{2}^{+} (\tau)) , \beta_{1} = \varphi^{+}_{2}\left(\beta_{2} \right)$ simply describe the portion of the ellipse $\mathcal{E}$ in\eqref{5234bis} contained in the quadrant: $\beta_1>4(N+1)$ and $\beta_2>4$.
\end{proof}

As already pointed out in Remark \ref{rmk23}, when $\tau \in (0, \tau^{(2)}_{0})$ we expect that problem $(P)_\tau$ should keep also the same uniqueness and non-degeneracy properties established in \cite{pot2} for cooperative systems, and which apply here 
for $\tau < 0.$

\begin{coro}\label{corA2} 
If $\tau \in [\tau^{(1)}_{1},1)$ then the conditions  \eqref{33}, \eqref{35} imply the condition \eqref{34} and they are necessary and sufficient for the solvability of $(P)_\tau$.
\end{coro}

\proof In this case, Theorem \ref{teoA} ensures the existence  of a (radial) solution for $(P)_\tau$ for any pairs $(\beta_1,\beta_2)$ such that $\beta_1 \in (\underline{\beta}_{1}(\tau),\overline{\beta}_{1}(\tau))$ and $\beta_{2}=\varphi^{+}_{1}(\beta_{1})$ (or equivalently $\beta_2 \in (\underline{\beta}_{2}(\tau),\overline{\beta}_{2}(\tau))$ and $\beta_{1}=\varphi^{+}_{2}(\beta_{2})$), which fulfil exactly the portion of $\mathcal{E}$ in\eqref{5234bis} subject to the constraint \eqref{35}. Furthermore, by Corollary \ref{42}
we see that the condition \eqref{34} is automatically satisfied in this case.

\end{proof}

\begin{remark}\label{rmk23a}
Since for  $\tau\in [\tau^{(1)}_{1},1)$ the condition \eqref{34} is ensured automatically by \eqref{33} and \eqref{35} which also guarantee the existence of a radial solution (see below), thus we suspect that, whenever solvable, problem $(P_{\tau})$ should admits \underline{only} radial solutions.
\end{remark}

To establish the existence of radial solutions as claimed in Theorem \ref{teoA}, we analyse the Cauchy problem: \eqref{39tris}-\eqref{312}. To be more precise, for $\alpha \in \mathbb{R}$ we let $(v_1 (r,\alpha), v_2 (r,\alpha))$ be the unique (global) solution of the Cauchy problem:

\begin{equation}\label{6112}
\begin{cases}
-(rv'_{1})'=r^{2N+1}e^{v_1}-\tau re^{v_2}, \quad r >0\\
-(rv'_{2})'=re^{v_2}-\tau r^{2N+1}e^{v_1}, \quad r >0\\
v_{1}(0)=\alpha \; \quad v'_{1}(0)=0\\
v_{2}(0)=0 \; \quad v'_{2}(0)=0
\end{cases}
\end{equation}

and set,

\begin{equation}\label{6113}
\beta_1 (\alpha) = \int^{\infty}_{0}r^{2N+1}e^{v_1}dr , \quad \beta_2 (\alpha) = \int^{\infty}_{0} re^{v_2}dr.
\end{equation} 

Clearly, the pair $(\beta_1 (\alpha), \beta_2 (\alpha))$ satisfy \eqref{33}, \eqref{34} and \eqref{35}, and in particular, $\beta_i (\alpha)$ is uniformly bounded with respect to $\alpha \in \mathbb{R},$ for $i=1,2.$

On the basis of Proposition \ref{527gir}, Proposition \ref{prop41} and their consequences, we define:

\begin{equation}\label{6113a}
\beta_{1,-\infty} (\tau) = 
\begin{cases}
\beta^{*}_{1}(\tau) \quad \text{ for } \tau \in (0, \tau^{(2)}_{1})\\
\underline{\beta}_{1}(\tau) \quad \text{ for } \tau \in [\tau^{(2)}_{1},1)
\end{cases}
\end{equation}

and,

\begin{equation}\label{6113b}
\beta_{2,-\infty} (\tau) = \varphi^{+}_{1}(\beta_{1,-\infty}(\tau)) =
\begin{cases}
4 \quad \text{ for } \tau \in (0, \tau^{(2)}_{0}]\\
\beta^{**}_{2}(\tau) \quad \text{ for } \tau \in (\tau^{(2)}_{0}, \tau^{(2)}_{1})\\
\overline{\beta}_{2}(\tau) \quad \text{ for } \tau \in [\tau^{(2)}_{1},1)
\end{cases}
\end{equation}

similarly, we let:

\begin{equation}\label{6113c}
\beta_{2,+\infty} (\tau) = 
\begin{cases}
\beta^{*}_{2}(\tau) \quad \text{ for } \tau \in (0, \tau^{(1)}_{1})\\
\underline{\beta}_{2}(\tau) \quad \text{ for } \tau \in [\tau^{(1)}_{1},1)
\end{cases}
\end{equation}

and,

\begin{equation}\label{6113d}
\beta_{1,+\infty} (\tau) = \varphi^{+}_{2}(\beta_{2,+\infty}(\tau)) =
\begin{cases}
4(N+1) \quad \text{ for } \tau \in (0, \tau^{(1)}_{0}]\\
\beta^{**}_{1}(\tau) \quad \text{ for } \tau \in (\tau^{(1)}_{0}, \tau^{(1)}_{1})\\
\overline{\beta}_{1}(\tau) \quad \text{ for } \tau \in [\tau^{(1)}_{1},1)
\end{cases}
\end{equation}

It is important to observe that,

\begin{equation}\label{6113*}
\beta_{i}^{-} (\tau) = \min \{ \beta_{i,-\infty} (\tau),  \beta_{i,+\infty} (\tau) \} \text{ and } \beta_{i}^{+} (\tau) = \max \{ \beta_{i,-\infty} (\tau),  \beta_{i,+\infty} (\tau) \}, \; i=1,2.
\end{equation}

The "existence" part of Theorem \ref{teoA} will be covered by the following:

\begin{theo}\label{teoB}
For all $\tau \in (0,1)$ there holds:
\begin{eqnarray}
\label{6114} (i) \; \lim_{\alpha \to +\infty} \beta_{1,\alpha} = \beta_{1,+\infty} (\tau) \quad \text{ and } \quad \lim_{\alpha \to +\infty} \beta_{2,\alpha} = \beta_{2,+\infty}\\
\label{6115} (ii) \lim_{\alpha \to -\infty} \beta_{1,\alpha} = \beta_{1,-\infty} (\tau) \quad \text{ and } \quad \lim_{\alpha \to -\infty} \beta_{2,\alpha} = \beta_{2,-\infty}
\end{eqnarray}
\end{theo}


By virtue of \eqref{612} and \eqref{147*}, we know already that Theorem \ref{teoB}  holds for the Toda system $\tau=1/2$; so we only need to establish \eqref{6114} and \eqref{6115} when,

$$
\tau \in (0,1) \; \text{ and } \tau \ne \frac{1}{2}.
$$

We start to analyse \eqref{6114}, and then handle \eqref{6115} similarly.  Actually, we are going to establish \eqref{6114} along any sequence: 

$$
\alpha_n \to +\infty,
$$
where we let (up to a subsequence):

\begin{equation}\label{6117}
\beta_{1,n} := \beta_1 (\alpha_n) \to \beta_1 \; \text{ and } \; \beta_{2,n} := \beta_2 (\alpha_n) \to \beta_2,
\end{equation}
with,

\begin{equation}\label{6119a}
\max \{ 4(N+1), \underline{\beta}_1 (\tau) \} \leq \beta_1 \leq \overline{\beta}_1 \; \text{ and } \beta_2 = \varphi_{1}^{+} (\beta_1),
\end{equation}

or equivalently, 

\begin{equation}\label{6119b}
\max \{ 4, \underline{\beta}_2 (\tau) \} \leq \beta_2 \leq \overline{\beta}_2 \; \text{ and } \beta_1 = \varphi_{2}^{+} (\beta_2).
\end{equation}
Also, to simplify notations, we let: 

$$
v_{1,n} (r) := v_1 (r, \alpha_n) \quad \text{ and } \quad v_{2,n} := v_2 (r,\alpha_n),
$$

and recall that,

$$v_{1,n} (0) = \alpha_n \to +\infty, \quad \text{ and } \quad v_{2,n} (0) =0.$$

Therefore, we see that $v_{1,n}$ admits a blow-up point at the origin, in the sense of Brezis-Merle \cite{bm};  see also \cite{bt, tar3} for more details.\\
 Actually, as a first important information, we prove that also $v_{2,n}$ must admit a blow-up point at the origin.

To this purpose, we recall a well known result established in \cite{bt} concerning blow-up sequences for  "`singular"' Liouville equations, which we formulate in a form convenient for our purposes.\\ 
Therefore, for $p>1$, we consider $v_n \in W^{2,p} (B_R)$ and $\sigma_n \in L^{p} (B_R)$, satisfying:

\begin{equation}\label{6119*} 
\begin{cases}
-\Delta v_n = |x|^{2N} e^{v_n} + \sigma_n \quad \text{ in } B_R\\
\sup_{B_R} v_n \to +\infty\\
(\int_{B_R} |x|^{2N} e^{v_n}dx) + \| \sigma_n \|_{L^{p} (B_R)}  \leq C\\
\sup_{\partial B_R} v_n - \inf_{\partial B_R} v_n \leq C
\end{cases}
\end{equation}

with $N > 0$ and a suitable constant $C>0$. Assume in addition that $v_n$ admits the origin as its only blow-up point in $B_R$, in the sense that:

\begin{equation}\label{6119**}
\forall \; \varepsilon > 0 \quad  \exists \; C_\varepsilon > 0 : \quad \sup_{\varepsilon \leq |x| \leq R} v_n \leq C_\varepsilon.
\end{equation}

We have:

\begin{theo}[\cite{bt}] \label{teobt}
Let $N > 0$ and $p>1$. If $v_n$ satisfies \eqref{6119*} and \eqref{6119**} then, along a subsequence, the following holds: 
$$|x|^{2N} e^{v_n} \rightharpoonup 8\pi(N+1)\delta_0  \quad \text { as } n \to +\infty,$$
weakly in the sense of measure in $B_R.$ 
\end{theo}

\qed

Going back to our sequence $ (v_{1,n}, v_{2,n}),$  we can use Lemma \ref{lem32} and Lemma \ref{harn}, in order to claim the following uniform estimates:

\begin{equation}\label{6120}
\forall \; r > 0: \quad v_{1,n} (r) + 2(N+1) \log r \leq C_0 \quad v_{2,n} (r) + 2 \log r \leq C_0,
\end{equation}

\begin{equation}\label{6121}
\forall \; 0 < r_0 < R_0 : \sup_{[r_0,R_0]} v_{j,n} (r) \leq \gamma \inf_{[r_0,R_0]} v_{j,n} (r) + C(1+ \log r_0 ), \quad j=1,2,
\end{equation}

with suitable  constants: $C_0=C_0(\tau,N)>0$, $\gamma=\gamma(\frac{r_0}{R_0}) \in (0,1)$ and $C=C(\tau, N,  \frac{r_0}{R_0})>0.$ \\

From \eqref{6120}, we see that $v_{i,n}$ attains its maximum value at some $ R_{i,n}\in [0, \infty),$ and therefore, 

\begin{equation}\label{6121a}
v_{i,n} (R_{i,n}) = \max_{0 \leq r < \infty} v_{i,n} (r), \quad i=1,2, 
\end{equation}

and, by taking also into account \eqref{6120}, we have:

\begin{equation}\label{6121b}
v_{1,n} (R_{1,n}) \to +\infty \; \text{ and } R_{1,n} \to 0, \quad \text{ as } n \to +\infty. 
\end{equation}

We show that the same property holds for $v_{2,n}$.

\begin{lemma}\label{lem4*}
The sequence $v_{2,n}$ admits a blow up point at the origin, and more precisely: 
$$v_{2,n} (R_{2,n}) \to +\infty \text{ and } R_{2,n} \to 0, \quad \text{ as } n \to +\infty.$$
\end{lemma}

\proof
We argue by contradiction and suppose that, for suitable $r_0>0$ and $C_0>0$ we have:

\begin{equation}\label{6122a}
v_{2,n} (r) \leq C_0, \quad \forall \, 0 < r < r_0.
\end{equation}
Furthermore, we can take $r_0>0$ smaller if necessarily, in order to ensure that, 

$$\int_{0}^{r_0} r e^{v_{2,n} (r)} \, dr \leq 2.$$

To proceed further, we recall the following consequence of the Alexandrov-Bol's inequality (see \cite{ban}), as pointed out by Suzuki in Proposition 4 of \cite{s}:

\begin{theo}\label{ab}
Let $\Omega\subset\mathbb{R}^2$ be a bounded domain with smooth boundary $\partial\Omega$. If $p\in C^{2}(\Omega)\cap C^{0}(\overline{\Omega})$ satisfies
\begin{equation}\label{logp}
-\Delta {\rm{log}}p\leq p\quad \mbox{ in }\quad\Omega,
\end{equation}
and $\Sigma:=\displaystyle{\int_{\Omega}p(x)dx}< 8\pi$, then we have:
\begin{equation}\label{alb}
\max_{\overline{\Omega}}p\leq\left(1-\frac{\Sigma}{8\pi}\right)^{-2}\max_{\partial\Omega}p.
\end{equation}
\end{theo}
\qed

Clearly, Theorem \ref{ab} applies for the radial function $p=e^{v_{2,n}}$ in $\Omega=B_{r_0},$ and for every $0< r < r_0,$ it implies that:

\begin{equation}\label{6122b}
1 = e^{v_{2,n} (0)} \leq \max_{B_r} e^{v_{2,n}} \leq \left( 1 - \frac{\int_{0}^{r} t e^{v_{2,n}}}{4} \right)^{-2} e^{v_{2,n} (r)} \leq 4 e^{v_{2,n} (r)}.
\end{equation}

In other words, by combining \eqref{6122a} and \eqref{6122b} with \eqref{6120} and \eqref{6121}, we conclude that,

\begin{equation}\label{6123}
\forall \, R>0 \quad \exists \, C_R>0 : \quad \| v_{2,n} \|_{L^\infty (B_R)} \leq C_R.
\end{equation}

As a consequence of \eqref{6123}, we are in position to apply Theorem \ref{teobt} to the sequence $v_{1,n},$ and along a subsequence, we conclude that:  $r^{2N} e^{v_{1,n}} \to 8 \pi (N+1) \delta_0$, weakly in the sense of measure on compact subsets of $\mathbb{R}^2$.

Furthermore, since $e^{v_{1,n}}$ is uniformly bounded away from the origin (see \eqref{6120}), by well known elliptic estimates, we also deduce that,

$$\forall \, \varepsilon >0 \quad \exists \, C_\varepsilon > 0 : \quad \| v_{2,n} \|_{C^{2,\alpha} (B_{1/\varepsilon} \setminus B_{\varepsilon}}) \leq C_\varepsilon.$$

Thus along a subsequence, $v_{2,n} \to v_2$ uniformly in $C^{2}_{loc} (\mathbb{R}^2 \setminus \{ 0 \})$, with  $v_2$ satisfying, in the sense of distributions, the following problem:

\begin{equation}
\begin{cases}
-\Delta v_2 = e^{v_2} - 8 \pi (N+1) \tau \delta_0, \quad \text{ in } \mathbb{R}^2\\
\| v_{2} \|_{L^\infty (B_1)} \leq C,
\end{cases}
\end{equation}

and this is clearly impossible. Therefore, by taking into account also \eqref{6120}, we may conclude that:  $v_{2,n} (R_{2,n}) \to +\infty  \;\text{ and } R_{2,n} \to 0,  \text{ as } n \to +\infty,$ as claimed. 

\qed

\begin{remark}\label{rmk61} 
A crucial ingredient in the proof of Lemma \ref{lem4*}, is given by the information that ${v_{2,n}}(0)=0$.  In fact,  we can check easily that  the same conclusion would hold under the more general assumption: $v_{2,n}(s_{n}) = C$ with 
$s_{n} \to 0,  \text{ as } n \to +\infty.$  Clearly,  the role between $v_{1,n}$ and $v_{2,n}$ can be interchanged.
\end{remark}

Next, for $R_{i,n} \geq 0$ defined in \eqref{6121a}, we let:

\begin{equation}\label{6124}
\varepsilon_{1,n} = e^{- \frac{v_{1,n} (R_{1,n})}{2(N+1)}} \quad , \quad \varepsilon_{2,n} = e^{- \frac{v_{2,n} (R_{2,n})}{2}}
\end{equation}

so that, from \eqref{6121b} and lemma \ref{lem4*} we know that,

\begin{equation}\label{6125}
R_{i,n} \to 0 \; \text{ and } \; \varepsilon_{i,n} \to 0, \quad \text{ as } n \to \infty; \quad i=1,2.
\end{equation}

We define:

$$v_{1,n}^{*} (r) = v_{1,n} ( \varepsilon_{2,n} r) + 2 (N+1) \log \varepsilon_{2,n},$$

and 

$$v_{2,n}^{*} (r) = v_{2,n} ( \varepsilon_{2,n} r) + 2 \log \varepsilon_{2,n}.$$

We know that, $(v_{1,n}^{*},v_{2,n}^{*})$ satisfies \eqref{6112} and,

\begin{equation}\label{6126}
\begin{split}
&\int_{0}^{+\infty} r^{2N+1} e^{v_{1,n}^{*} (r)} \, dr = \int_{0}^{+\infty} r^{2N+1} e^{v_{1,n} (r)}dr = \beta_{1,n} \to \beta_1 \quad \text{ as } n \to \infty, \\
&\int_{0}^{+\infty} r e^{v_{2,n}^{*} (r)} \, dr = \int_{0}^{+\infty} r e^{v_{2,n} (r)} \, dr = \beta_{2,n} \to \beta_2 \quad \text{ as } n \to \infty.
\end{split}
\end{equation}

Furthermore, we notice that:

\begin{eqnarray}
 \label{6127}& \max_{r \geq 0} [v_{1,n}^{*}(r)] = v_{1,n}^{*} \left( \frac{R_{1,n}}{\varepsilon_{2,n}} \right) = 2 (N+1) \log \left( \frac{\varepsilon_{2,n}}{\varepsilon_{1,n}} \right), \\
\label{6128} &\max_{r \geq 0}[ v_{2,n}^{*}(r)] = v_{2,n}^{*} \left( \frac{R_{2,n}}{\varepsilon_{2,n}} \right) = 0. 
\end{eqnarray}

We have:

\begin{lemma}\label{lem2}
Along a subsequence, the following holds as $n \to \infty:$

\begin{equation}\label{6129}
\frac{R_{1,n}}{R_{2,n}} \to 0 \; , \quad \frac{\varepsilon_{1,n}}{\varepsilon_{2,n}} \to 0 \; , \quad \frac{R_{1,n}}{\varepsilon_{1,n}} \to 0 \; \text{ and } \quad \frac{R_{2,n}}{\varepsilon_{2,n}} \to l
\end{equation}
with suitable $l>0$. In particular $v_{1,n}^{*}$ admits a blow up point at the origin while $v_{2,n}^{*}$ is uniformly bounded in $C^{2,\alpha}_{loc} (\mathbb{R}^2 \setminus \{ 0 \})$.
\end{lemma}

\proof From Lemma \ref{lem4*} we derive in particular that $R_{2,n}> 0,$ and therefore we can use  the second equation in \eqref{6112} to find,

\begin{equation}\label{6130}
\begin{split}
0 &\leq v_{2,n} (R_{2,n}) - v_{2,n} (0) = - \int_{0}^{R_{2,n}} \frac{1}{r} \int_{0}^{r} (t e^{v_{2,n} (t)} - \tau t^{2N+1} e^{v_{1,n} (t)}) \, dt\\
&\leq \tau \int_{0}^{R_{2,n}} \frac{1}{r} \int_{0}^{r}  t^{2N+1} e^{v_{1,n} (t)} \leq \frac{\tau R_{2,n}^{2(N+1)}}{(2(N+1))^2} e^{v_{1,n} (R_{1,n})}
\end{split}
\end{equation}

Thus, if we recall that: $v_{2,n} (0) = 0$, $R_{1,n}^{2(N+1)} e^{v_{1,n} (R_{1,n})} \leq C$ (see \eqref{6120}) and $v_{2,n}(R_{2,n}) \to +\infty$ as $n \to +\infty$, from \eqref {6130} we conclude that, 

\begin{equation}\label{6131}
\frac{R_{1,n}^{2(N+1)}}{R_{2,n}^{2(N+1)}} \to 0, \quad \text{ as } n \to +\infty;
\end{equation}

\begin{equation}\label{6132}
\left( \frac{R_{2,n}}{\varepsilon_{1,n}}\right)^{2(N+1)} = R_{2,n}^{2(N+1)} e^{v_{1,n} (R_{1,n})} \to +\infty, \quad \text{ as } n \to +\infty.
\end{equation}

Furthermore, by  \eqref{6132} and the fact that $R_{2,n}^{2} e^{v_{2,n} (R_{1,n})} \leq C$ (see \eqref{6120}), we also deduce that,

$$
\left( \frac{\varepsilon_{2,n}}{\varepsilon_{1,n}} \right)^{2} = \left( \frac{\varepsilon_{2,n}}{R_{2,n}}\right)^{2} \left( \frac{R_{2,n}}{\varepsilon_{1,n}}\right)^{2} = 
\left( \frac{R_{2,n}}{\varepsilon_{1,n}}\right)^{2} \frac{1}{(R_{2,n})^2 e^{v_{2,n} (R_{2,n})}} \geq \frac{1}{C} \left( \frac{R_{2,n}}{\varepsilon_{1,n}}\right)^{2} \to +\infty, \quad \text{ as } n \to +\infty.
$$

In other words, as $n \to +\infty$ there holds:

$$\frac{\varepsilon_{2,n}}{\varepsilon_{1,n}} \to +\infty,$$ 

and so, 

$$v_{1,n}^{*} \left( \frac{R_{1,n}}{\varepsilon_{2,n}} \right) = 2 (N+1) \log \left( \frac{\varepsilon_{2,n}}{\varepsilon_{1,n}} \right) \to +\infty,$$

$$\frac{R_{1,n}}{\varepsilon_{2,n}}  \leq \; \frac{R_{1,n}}{R_{2,n}} R_{2,n} e^{v_{2,n}(R_{2,n})/2} \leq C \; \frac{R_{1,n}}{R_{2,n}} \to 0.$$

Thus, we have shown that,  $v_{1,n}^{*}$ admits a blow up point at the origin, and by the estimates \eqref {6120} \eqref {6121}, which continue to hold for the pair $(v_{1,n}^{*}, v_{2,n}^{*})$, we see that actually the origin is the only blow-up point for 
$v_{1,n}^{*}.$ \\
On the other hand, from \eqref{6128} we see that  $v_{2,n}^{*}$ is uniformly bounded from above, and $v_{2,n}^{*}(\frac{R_{2,n}}{\varepsilon_{2,n}})=0.$ At this point, by recalling that:  $\frac{R_{2,n}}{\varepsilon_{2,n}} < C$  (see \eqref {6120}), we can use 
Remark \ref{rmk61} in order to conclude that $\frac{R_{2,n}}{\varepsilon_{2,n}}$ must be also bounded from below away from zero. So, along a subsequence, there holds:

\begin{equation}\label{444}
\frac{R_{2,n}}{\varepsilon_{2,n}} \to l > 0, \text{ as } n \to +\infty.
\end{equation}
 
Consequently, we can use \eqref{444} together with the Harnack-type inequality \eqref {6121}, in order to conclude that, for all $\varepsilon >0$ sufficiently small,  $|v_{2,n}^{*}|$ is uniformly bounded in $ [\varepsilon, \frac{1}{\varepsilon}]$.
Since also $v_{1,n}^{*}$ is uniformly bounded from above in $ [\varepsilon, \frac{1}{\varepsilon}]$,  we can use standard elliptic estimates  to obtain that $v_{2,n}^{*}$ is actually uniformly bounded in $C^{2,\alpha}_{loc} (0,+\infty)$ as claimed.

Finally, it remains to show that (along a subsequence) $\frac{R_{1,n}}{\varepsilon_{1,n}}$  tends to zero. To this purpose, we recall first that, by \eqref {6120}, we have:  $\frac{R_{1,n}}{\varepsilon_{1,n}} < C.$ 
Therefore, if we define:

$$\xi_{1,n} (r) = v_{1,n}^{*} \left( \frac{\varepsilon_{1,n}}{\varepsilon_{2,n}} r \right) + 2 (N+1) \log \left( \frac{\varepsilon_{1,n}}{\varepsilon_{2,n}} r \right) = v_{1,n} ( \varepsilon_{1,n} r) + 2 (N+1) \log \varepsilon_{1,n},$$
$$\xi_{2,n} (r) = v_{2,n}^{*} \left( \frac{\varepsilon_{1,n}}{\varepsilon_{2,n}} r \right) + 2 \log \left( \frac{\varepsilon_{1,n}}{\varepsilon_{2,n}} r \right) = v_{2,n} ( \varepsilon_{1,n} r) + 2 \log \varepsilon_{1,n},$$

we see that $(\xi_{1,n} (r), \xi_{2,n} (r))$ satisfies \eqref{6112} together with the estimates \eqref {6120} \eqref {6121}.\\
 Furthermore,   

$$e^{\xi_{2,n} (r)} = \left( \frac{\varepsilon_{1,n}}{\varepsilon_{2,n}} \right)^2 e^{v_{2,n}^{*} \left( \frac{\varepsilon_{1,n}}{\varepsilon_{2,n}} r \right)} \to 0 \text{ uniformly in } [0,+\infty),$$

while, $\xi_{1,n}$ satisfies:

\begin{equation*}
\begin{cases}
-\Delta \xi_{1,n} = r^{2N} e^{\xi_{1,n}} - \tau e^{\xi_{2,n}}, \quad r>0\\
\max_{r \geq 0} \xi_{1,n}=\xi_{1,n} \left( \frac{R_{1,n}}{\varepsilon_{1,n}} \right) = 0 \\
\int_{0}^{+\infty} r^{2N+1} e^{\xi_{1,n} (r)} \, dr \leq C,
\end{cases}
\end{equation*}

with  a suitable  constant $C>0$.  Therefore, $\xi_{1,n}$ is locally uniformly bounded (see \eqref{6121}), and by standard elliptic estimates and a diagonalization process, we obtain (along a subsequence): 

\begin{equation}\label{6133}
\xi_{1,n} \to \xi \; \text{ in } C^{2}_{loc},\quad \text{ and }\quad \frac{R_{1,n}}{\varepsilon_{1,n}}\to R_0, \quad n \to +\infty,
\end{equation}

with $\xi = \xi (r)$ the \underline{unique} radial solution of the following  "singular" Liouville equation:

\begin{equation*}
\begin{cases}
-\Delta \xi = r^{2N} e^{\xi}\\
 \max_{r \geq 0} \xi_{1}=\xi_{1}(R_{0}) = 0\\
\int_{0}^{+\infty} r^{2N+1} e^{\xi (r)} \, dr \leq C.
\end{cases}
\end{equation*}

From the classification result in \cite{pt}, we know that, 

\begin{equation}\label{6134}
\xi (r) = \log \left( \frac{1}{1+\frac{r^{2(N+1)}}{8(N+1)^2}} \right)^2 \; \text{ and } \; \int_{0}^{+\infty} r^{2N+1} e^{\xi (r)} \, dr = 4 (N+1)
\end{equation}

and we can use \eqref{6133} and \eqref{6134}, in order to we derive in particular that $R_0 = 0$, that is:

$$
\frac{R_{1,n}}{\varepsilon_{1,n}} \to 0,
$$

as claimed. 

\qed

\textbf{The proof of Theorem \ref{teoB}}

Since $v_{2,n}^{*} (r) \leq C,$ we can use Lemma \ref{lem2}, together with Theorem \ref{teobt} in order to conclude that, along a subsequence, the following holds : 

\begin {equation}\label{6137a}
r^{2N} e^{v_{1,n}^{*}} \to 8 \pi (N+1) \delta_0 \quad \text { weakly in the sense of measures on compact sets of } \mathbb{R}^{2}, 
\end{equation}
and 
\begin{equation}\label{6137b}
v_{2,n}^{*} \to v_{2}^{*} \, \text{ in } C^{2}_{loc} (\mathbb{R}^{2} \setminus \{ 0 \}),
\end{equation}

as $n \to \infty$, and $v_{2}^{*} $ defines a (radial) solution (in the sense of distributions) of  the following problem:

\begin{equation}\label{6137}
\begin{cases}
-\Delta v_{2}^{*} = e^{v_{2}^{*}} - 8 \pi (N+1)\tau \delta_0, \text{ in }\mathbb{R}^{2}\\
\int_{\mathbb{R}^{2}} e^{v_{2}^{*}} dx \ < \infty.
\end{cases}
\end{equation}

By the fact that  $\max_{r \geq 0}  \; v_{2}^{*}(r) = 0,$ we can write down the explicit expression for $v_{2}^{*} = v_{2}^{*} (r)$   in terms of the limit value $l$ in \eqref{6129},  In particular we know that,

$$
\int_{0}^{+\infty}  e^{v_{2}^{*} (r)} \, r dr = \beta_{2}^{*} (\tau) = 4  + 8 \tau(N+1):=\beta_{2}^*(\tau),
$$

see \cite{pt} or Theorem 2.2.1 in \cite{tar3}).

Consequently, from Fatou's lemma we deduce that,

\begin{equation}\label{6138}
\lim_{n \to \infty} \int_{0}^{+\infty}  e^{v_{2,n}^{*} (r)} \, rdr = \beta_{2} \geq \beta_{2}^{*} (\tau).
\end{equation}

While, by means of Lebesgue dominated convergence, we have: 
$$\forall \, \varepsilon > 0 \; \exists \; R_{\varepsilon} > 0 \text{ and } \; n_{\varepsilon} \in \mathbb{N}:$$

\begin{equation}\label{6139}
\int_{0}^{R} r e^{v_{2,n}^{*} (r)} \, dr \geq \beta_{2}^{*} (\tau) - \varepsilon, \quad \forall R \geq R_{\varepsilon} \text{ and } \; \forall n \geq n_{\varepsilon}.
\end{equation}

\begin{equation}\label{6140}
\textbf{ Claim 1: } \text{ if } \tau \in (0,\tau_{0}^{(1)}] \text{ then the pair } (\beta_1, \beta_2) \text{ satisfies: } \; \beta_1 = 4(N+1) \text{ and } \beta_2 = \beta_{2}^{*}(\tau).
\end{equation}

Indeed for $\tau \in (0,\tau_{0}^{(1)}],$  we can use  \eqref{coro5282} to obtain that,

\begin{equation}\label{6141}
\varphi_{2}^{+} (\beta_{2}^{*}) = 4 (N+1) \leq \beta_1 = \varphi_{2}^{+} (\beta_{2}) \leq \varphi_{2}^{+} (\beta_{2}^{*}),
\end{equation}

where, the last inequality follows by  \eqref{6138} and the monotonicity of $\varphi_{2}^{+}$. Clearly,  \eqref{6140} easily follows by \eqref{6141}.

\begin{equation}\label{6142}
\textbf{ Claim 2: } \text{ if } \tau \in (\tau_{0}^{(1)},\tau_{1}^{(1)}) \text{ then } \beta_1 = \beta_{1}^{**} \text{ and } \beta_{2} = \beta_{2}^{*}
\end{equation}

To establish \eqref{6142}, we notice that from \eqref{corocoro110} in this case there holds:

\begin{equation}\label{6143}
\beta_{1}^{**} (\tau) = \varphi_{2}^{+} (\beta_{2}^{*} (\tau)) \geq \varphi_{2}^{+} (\beta_{2}) = \beta_1. 
\end{equation}

While, from part $(i)$ of Proposition \ref{prop41} we have:

\begin{equation}\label{6144}
\beta_2 - \tau \beta_1 \geq \beta_{2}^{*} (\tau) - \tau \beta_{1}^{**} (\tau) > 2.
\end{equation}

We fix $\varepsilon_0 > 0$ sufficiently small, so that: 

\begin{equation}\label{6145}
\beta_{2}^{*} (\tau) - \tau \beta_{1}^{**} (\tau) > 2 (1 + \varepsilon_0),
\end{equation}

and then, we let $n_0 \in \mathbb{N}$ and $R_0 > 1$ sufficiently large, so that  $\forall \, n \geq n_0$ there holds,

\begin{equation}\label{6146}
\int_{0}^{R_0} r e^{v_{2,n}^{*} (r)} \, dr \geq \beta_{2}^{*} - \varepsilon_0,
\end{equation}

(see \eqref{6139}), and

\begin{equation}\label{6147}
\int_{0}^{\infty} r^{2N+1} e^{v_{1,n}^{*} (r)} \, dr < \beta_{1}^{**} + \varepsilon_0,
\end{equation}

(see \eqref{6143}).

As a consequence, for every $r > R_0$ and $n \geq n_0$ we can estimate:

\begin{equation}\label{6147bis}
\begin{split}
&v_{2,n}^{*} (r) - v_{2,n}^{*} (R_{0}) = - \int_{R_0}^{r} \frac{1}{t} \left( \int_{0}^{t} (s e^{v_{2,n}^{*} (s)} - \tau s^{2N+1} e^{v_{1,n}^{*} (s)}) \, ds \right)\, dt\\
&= - (\log r) \int_{0}^{r} (s e^{v_{2,n}^{*} (s)} - \tau s^{2N+1} e^{v_{1,n}^{*} (s)}) \, ds + (\log R_0) \int_{0}^{R_0} (s e^{v_{2,n}^{*} (s)} - \tau s^{2N+1} e^{v_{1,n}^{*} (s)}) \, ds\\
&+ \int_{R_0}^{r} (\log s) (s e^{v_{2,n}^{*} (s)} - \tau s^{2N+1} e^{v_{1,n}^{*} (s)}) \, ds.
\end{split}
\end{equation}

Thus, from \eqref{6147bis},  we obtain:

\begin{equation}\label{6148}
v_{2,n}^{*} (r) - v_{2,n}^{*} (R_{0}) \leq - (\log\frac{ r}{R_0})( \int_{0}^{R_0} s e^{v_{2,n}^{*} (s)}ds - \tau \int_{0}^{r}s^{2N+1} e^{v_{1,n}^{*} (s)} \, ds).
\end{equation}

Therefore, by recalling that $v_{2,n}^{*} (R_{0}) \leq 0$, we can use \eqref{6146} and \eqref{6147} and \eqref{6148} to find a suitable constant $C_0 > 0$ such that ,

\begin{equation}\label{6149}
v_{2,n}^{*} (r) \leq - (\log r) (\beta_{2}^{*} (\tau) - \tau \beta_{1}^{**} (\tau) - (1+\tau) \varepsilon_0) + C_0, \quad\text{ for } n \geq n_\varepsilon,\quad r \geq R_0,
\end{equation}
where, form \eqref{6145} we know that,

\begin{equation}\label{6150}
(\beta_{2}^{*} (\tau) - \tau \beta_{1}^{**} (\tau) - (1+\tau) \varepsilon_0) > 2+(1-\tau)\varepsilon_0 > 2.  
\end{equation}

As a consequence of \eqref{6149} and \eqref{6150}, we see that, for every $\varepsilon >0$ we can find $R_\varepsilon > 0$ and $n_\varepsilon \in \mathbb{N}$ such that,

\begin{equation*}
\int_{R_\varepsilon}^{+\infty} r e^{v_{2,n}^{*} (r)} \, dr \leq \frac{\varepsilon}{2}, \quad \forall \, n \geq n_\varepsilon
\end{equation*}

By taking into account that,

\begin{equation}\label{6150bis}
\lim_{n \to +\infty} \int_{0}^{R_\varepsilon} r e^{v_{2,n}^{*} (r)} \, dr = \int_{0}^{R_\varepsilon} r e^{v_{2}^{*} (r)} \, dr \leq  \int_{0}^{+\infty} r e^{v_{2}^{*} (r)} \, dr = \beta_{2}^{*},
\end{equation}

and \eqref{6138}, we can take $n_\varepsilon $ larger if necessary, so that $\forall \, n \geq n_\varepsilon$ the following holds:

\begin{equation}\label{6150a}
\int_{0}^{R_\varepsilon} r e^{v_{2,n}^{*} (r)} \, dr \leq \; \beta_{2}^{*} + \frac{\varepsilon}{2}, \quad \text{ and } \quad \beta_{2}^{*} - \varepsilon < \int_{0}^{+\infty} r e^{v_{2,n}^{*} (r)} \, dr.
\end{equation}
In conclusion, for every $\varepsilon > 0,$ we have shown that,
\begin{equation}\label{6150b}
\beta_{2}^{*} - \varepsilon <\int_{0}^{+\infty} r e^{v_{2,n}^{*} (r)} \, dr  < \beta_{2}^{*} + \varepsilon, \quad \forall \, n \geq n_\varepsilon ;
\end{equation}

and so: $\beta_{2}^{*} = \beta_{2}$ and $\beta_1 = \varphi_{2}^{+} (\beta_{2}^{*}) = \beta_{1}^{**}$ as claimed. \\

\begin{equation}\label{6151}
\textbf{ Claim 3: } \text{ if } \tau \in [\tau_{1}^{(1)},1) \text{ then } \beta_2 = \underline{\beta}_{2} \text{ and } \beta_{1} = \overline{\beta}_{1}
\end{equation}

The case $\tau = \tau_{1}^{(1)}$ follows by a limiting argument form Claim 2. Indeed, as $\tau \nearrow \tau_{1}^{(1)}$, we have: $\beta_{2}^{*} (\tau) \to \beta_{2}^{*} (\tau_{1}^{(1)}) = \underline{\beta}_{2} (\tau_{1}^{(1)})$ and $\beta_{1}^{**} (\tau) \to \beta_{1}^{**} (\tau_{1}^{(1)}) = \overline{\beta}_{1} (\tau_{1}^{(1)})$. So it will suffice to prove \eqref{6142} when $\tau \in (\tau_{1}^{(1)},1).$

Since in this case: $\beta_{2}^{*} (\tau) - \tau \beta_{1}^{**} (\tau) < 2$, then the inequality \eqref{6138} must hold with the strict sign, that is:
\begin{equation}\label{6152} 
 \beta_2 > \beta_2^{*}.
 \end{equation}

To establish \eqref{6151}, we need to show that $\beta_{2,\tau} = \beta_2 - \tau \beta_1 = 2$. To this purpose, we recall that $\beta_2 - \tau \beta_1 \geq 2$, and so arguing by contradiction, we suppose that actually: $\beta_{2,\tau}=\beta_2 - \tau \beta_1 > 2.$

Let us define (via Kelvin transform):
$$
\hat{v}_{i,n}(r)=v^{*}_{i,n}\left(\dfrac{1}{r}\right)+\beta^{n}_{i,\tau}{\rm{log}}\left(\dfrac{1}{r}\right) \quad i=1, 2,
$$
which satisfies:
\begin{equation}\label{kelvin}
\begin{cases}
-\Delta\hat{v}_{1,n}=r^{\beta^{n}_{1,\tau}-2(N+2)}e^{\hat{v}_{1,n}}-\tau r^{\beta^{n}_{2,\tau}-4}e^{\hat{v}_{2,n}}\\
-\Delta\hat{v}_{2,n}=r^{\beta^{n}_{2,\tau}-4}e^{\hat{v}_{2,n}}-\tau r^{\beta^{n}_{1,\tau}-2(N+2)}e^{\hat{v}_{1,n}}\\
\end{cases}
\end{equation}
with
\begin {equation} \label{kalvin1}
\beta^{n}_{1,\tau} = \beta_{1,n} -\tau \beta_{2,n} \quad \text { and } \quad \beta^{n}_{2,\tau} = \beta_{2,n} -\tau \beta_{1,n},
\end {equation}
and 
\begin{equation}\label{kelvinbis}
\int^{\infty}_{0}s^{\beta^{n}_{1,\tau}-2(N+2)}e^{\hat{v}_{1,n}(s)}ds = \beta_{1, n} \quad \text{ and } \quad \int^{\infty}_{0}s^{\beta^{n}_{2,\tau}-4}e^{\hat{v}_{2,n}(s)}ds = \beta_{2,n}.
\end{equation}

  
We set,
$$
\beta^{\infty}_{1}:=\lim_{r\rightarrow0}(\liminf_{n\rightarrow\infty}\int^{r}_{0}s^{\beta^{n}_{1,\tau}-2(N+2)}e^{\hat{v}_{1,n}}ds),
$$
and
$$
\beta^{\infty}_{2}:=\lim_{r\rightarrow0}(\liminf_{n\rightarrow\infty}\int^{r}_{0}s^{\beta^{n}_{2,\tau}-4}e^{\hat{v}_{2,n}}ds).
$$
Recalling that, $\hat{v}_{2,n}$ is uniformly bounded in $ C^{2,\alpha}_{loc}(\mathbb{R}^2\setminus\left\{0\right\}),$ while for every compact set $\Omega\subset\subset \mathbb{R}^2\setminus\left\{0\right\}$ we have: $\displaystyle{\max_{\Omega}\hat{v}_{1,n}\rightarrow-\infty},$ we may deduce that, along a  subsequence, the following holds: 
$$\hat{v}_{2,n}\rightarrow\hat{v}_2 \quad \text{ uniformly in} \; C^{2}( \mathbb{R}^2\setminus\left\{0\right\}),$$
and 
$$
r^{\beta^{n}_{1,\tau}-2(N+2)}e^{\hat{v}_{1,n}}\rightharpoonup\beta^{\infty}_{1}\delta_{0} \quad \text{weakly in the sense of measures, locally in}\, \mathbb{R}^2.
$$
as $n\to \infty.$ 
Furthermore, as $\beta_{2,\tau} > 2$, we have also that,
$$
r^{\beta^{n}_{2,\tau}-4}e^{\hat{v}_{2,n}}\rightharpoonup\beta^{\infty}_{2}\delta_{0}+r^{\beta_{2,\tau}-4}e^{\hat{v}_2}  \quad \text{ as } n\to \infty,
$$
weakly in the sense of measure on compact set of $\mathbb{R}^2.$
In particular, $ \hat{v}_2$ defines a radial solution for the problem:
\begin{equation}\label{5*}
\begin{cases}
-\Delta\hat{v}_{2}=|x|^{\beta_{2,\tau} - 4}e^{\hat{v}_{2}}+(\beta^{\infty}_{2}-\tau\beta^{\infty}_{1})\delta_{0} \quad \text{ in } \mathbb{R}^2\\
\int_{\rr^{2}}|x|^{\beta_{2,\tau} - 4}e^{\hat{v}_{2}} < \infty.
\end{cases}
\end{equation}
As already mentioned, all solutions of \eqref{5*} are completely classified in \cite{pt} and in particular, in order to fulfill the integrability condition, we  must have:

\begin{equation}\label{cbe}
\beta_{2,\tau}-(\beta^{\infty}_{2}-\tau\beta^{\infty}_{1})>2.
\end{equation}

To proceed further, we fix $\varepsilon>0$ sufficiently small and  let $n_{\varepsilon}\in \nn$ and $R_{\varepsilon}>>1$ sufficiently large so that, for any $n\geq n_{\varepsilon}$ and $R\geq R_{\varepsilon}$ the following holds:
\begin{equation}\label{intg}
 \beta^{n}_{2,\tau} > \beta_{2,\tau} - \dfrac{\varepsilon}{4}, \quad \;  \int^{\infty}_{R}se^{v^{*}_{2,n}}ds\leq\beta^{\infty}_{2}+\dfrac{\varepsilon}{4}\quad\mbox{ and }\quad
\int^{\infty}_{R}s^{2N+1}e^{v^{*}_{1,n}}ds\geq\beta^{\infty}_{1}-\dfrac{\varepsilon}{4\tau}.
\end{equation}

Therefore, for $r>R_{\varepsilon}$ and $n\geq n_{\varepsilon}$ we obtain:
\begin{equation}
\begin{aligned}
v&^{*}_{2,n}(r)-v^{*}_{2,n}(R_{\varepsilon})=-\int^{r}_{R_{\varepsilon}}\dfrac{1}{t}(\int^{t}_{0}se^{v^{*}_{2,n}}ds)dt+\tau
\int^{r}_{R_{\varepsilon}}\dfrac{1}{t}(\int^{t}_{0}s^{2N+1}e^{v^{*}_{1,n}}ds)dt\\
&=\left[-{\rm{log}}\,t\int^{t}_{0}se^{v^{*}_{2,n}}ds\right]^{t=r}_{t=R_{\varepsilon}}+\int^{r}_{R_{\varepsilon}} t {\rm{log}}\,te^{v^{*}_{2,n}}dt+\tau\left[{\rm{log}}\,t \int^{t}_{0}s^{2N+1}e^{v^{*}_{1,n}}ds\right]^{t=r}_{t=R_{\varepsilon}}-\tau\int^{r}_{R_{\varepsilon}}t^{2N+1}{\rm{log}}\,t e^{v^{*}_{1,n}}dt\\
&\leq-{\rm{log}}\,r\left(\int^{r}_{0}se^{v^{*}_{2,n}}ds-\tau\int^{r}_{0}s^{2N+1}e^{v^{*}_{1,n}}ds\right)+\int^{r}_{R_\varepsilon}{\rm log}\,s\left(se^{v^{*}_{2,n}}-\tau s^{2N+1}e^{v^{*}_{1,n}}\right)ds+C_\varepsilon
\end{aligned}
\end{equation}
with $C_{\varepsilon}>0$ a suitable constant depending on $\varepsilon$ only.\\
Therefore, by recalling that $\beta^{n}_{2,\tau}=\int^{\infty}_{0}\left(se^{v^{*}_{2,n}}-\tau s^{2N+1}e^{v^{*}_{1,n}}\right)ds$ we find:
\begin{equation}
\begin{aligned}
v&^{*}_{2,n}(r)-v^{*}_{2,n}(R_{\varepsilon})+\beta^{n}_{2,\tau}{\rm log}\,r\\
&={\rm log}\,r\int^{\infty}_{r}se^{v^{*}_{2,n}}ds-\tau{\rm log}\,r\int^{\infty}_{r}s^{2N+1}e^{v^{*}_{1,n}}+\int^{r}_{R_{\varepsilon}}s{\rm log}\,s e^{v^{*}_{2,n}}-\tau\int^{r}_{R_{\varepsilon}}s^{2N+1}{\rm log}\,s e^{v^{*}_{1,n}}ds+C_{\varepsilon}\\
&\leq{\rm log}\,r\int^{\infty}_{r}se^{v^{*}_{2,n}}ds-\tau{\rm log}\,r\int^{\infty}_{r}s^{2N+1}e^{v^{*}_{1,n}}+{\rm log}\,r\int^{r}_{R_{\varepsilon}}s e^{v^{*}_{2,n}}-\tau\int^{r}_{R_{\varepsilon}}s^{2N+1}{\rm log}\,s e^{v^{*}_{1,n}}ds+C_{\varepsilon}\\
&={\rm log}\,r\int^{\infty}_{R_{\varepsilon}}se^{v^{*}_{2,n}}ds-\tau{\rm log}\,r\int^{\infty}_{r}s^{2N+1}e^{v^{*}_{1,n}}-\tau\int^{r}_{R_{\varepsilon}}s^{2N+1}{\rm log}\,s e^{v^{*}_{1,n}}ds+C_{\varepsilon}\\
&\leq{\rm log}\,r\left[\int^{\infty}_{R_{\varepsilon}}se^{v^{*}_{2,n}}ds-\tau\int^{\infty}_{r}s^{2N+1}e^{v^{*}_{1,n}}\right]+C_{\varepsilon}
\end{aligned}
\end{equation}
and by means of  \eqref{intg}, we derive:
$$
v^{*}_{2,n}(r)\leq-{\rm log}\,r\left[\beta_{2,\tau}-(\beta^{\infty}_{2}-\tau\beta^{\infty}_{1})-\varepsilon\right]+C_{\varepsilon}
$$
with $C_{\varepsilon}>0$ a suitable constant depending on $\varepsilon$ only.
At this point, we can use \eqref{cbe}, to obtain that, for $\varepsilon > 0$ sufficiently small, there holds: 
\begin{equation}\label{x}
v^{*}_{2,n}(r)\leq-({\rm log}\,r)\left(2+\varepsilon\right)+C_{\varepsilon}, \quad \text{ for } r\geq R_\varepsilon \quad \text{ and } \;  n\geq n_\varepsilon.
\end{equation}
But this is impossible, since as above it yields that: $\beta_2 = \beta_2^{*}$,  in contradiction to \eqref{6152}.


Thus we must have that necessarily: $\beta_2 - \tau \beta_1 = 2$, that is $\beta_2 = \underline{\beta}_2 (\tau)$ and $\beta_1 = \overline{\beta}_1 (\tau),$ and  also Claim 3 is established.

Since Claim 1, 2 and 3 hold along any sequence, clearly they imply that part $ (i)$ of Theorem \ref{teoB} holds.

In order to establish part $(ii)$ we argue similarly, only now  we take a sequence: 

\begin{equation}\label{6154}
\alpha_n \to -\infty
\end{equation}

and as above, we let $v_{i,n} (r) = v_i(r, \alpha_n)$, $i=1,2$ and suppose that,

$$\beta_{1,n} := \int_{0}^{+\infty} s^{2N+1} e^{v_{1,n} (s)} \, ds \to \beta_1 \text{ and } \beta_{2,n} := \int_{0}^{+\infty} s e^{v_{2,n} (s)} \, ds \to \beta_2$$

with $\beta_1$ and $\beta_2$ satisfying \eqref{6119a} and \eqref{6119b}.

In view of \eqref{6154}, it is convenient to let, 
$$t_n = e^{-\frac{\alpha_n}{2(N+1)}} \to + \infty, \quad \text{ as } n\to \infty, $$ 
and consider the following rescaled version of $(v_{1,n} , v_{2,n})$ :
\begin{equation}\label{6155}
\begin{split}
&\hat{v}_{1,n} (r) = v_{1,n} (t_n r) - v_{1,n} (0) = v_{1,n} (t_n r) - \alpha_n = v_{1,n} (t_n r) + 2 (N+1) \log t_n,\\
&\hat{v}_{2,n} (r) = v_{2,n} (t_n r) + 2 \log t_n;
\end{split}
\end{equation}

Thus, $(\hat{v}_{1,n}, \hat{v}_{2,n})$ is a solution of the system of ODE's in \eqref{6112} and it satisfies:

\begin{equation}\label{156}
\begin{split}
&\int_{0}^{+\infty} s^{2N+1} e^{\hat{v}_{1,n} (s)} \, ds  = \int_{0}^{+\infty} s^{2N+1} e^{v_{1,n} (s)} \, ds = \beta_{1,n} \to \beta_1\\
&\int_{0}^{+\infty} s e^{\hat{v}_{2,n} (s)} \, ds  = \int_{0}^{+\infty} s e^{v_{2,n} (s)} \, ds = \beta_{2,n} \to \beta_2\\
&\hat{v}_{1,n} (0) = 0 \text{ and } \hat{v}_{2,n} (0) = 2 \log t_{n} \to +\infty
\end{split}
\end{equation}

as $n \to \infty.$ Therefore, the arguments above can be applied to  $(\hat{v}_{1,n}, \hat{v}_{2,n})$, simply with the role of $v_{1,n}$ now played by $\hat{v}_{2,n}$ and that of $v_{2,n}$  played by $\hat{v}_{1,n}$.

Hence, with the obvious modifications, one can carry out the same blow-up analysis and  as above (along a subsequence) arrive at the following conclusion:

\begin{equation*}
\int_{0}^{+\infty} s^{2N+1} e^{\hat{v}_{1,n} (s)} \, ds  = \int_{0}^{+\infty} s^{2N+1} e^{v_{1,n} (s)} \, ds \to \beta_{1,-\infty} = \begin{cases}
\beta^{*}_{1}(\tau) \quad \text{ for } \tau \in (0, \tau^{(2)}_{1})\\
\underline{\beta}_{1}(\tau) \quad \text{ for } \tau \in [\tau^{(2)}_{1},1)
\end{cases}
\end{equation*}
while,
\begin{equation*}
\int_{0}^{+\infty} s e^{\hat{v}_{2,n} (s)} \, ds  = \int_{0}^{+\infty} s e^{v_{2,n} (s)} \, ds \to \varphi_{1}^{+} (\beta_1) = \beta_{2,-\infty} = \begin{cases}
4 \quad \text{ for } \tau \in (0, \tau^{(2)}_{0})\\
\beta^{**}_{2}(\tau) \quad \text{ for } \tau \in (\tau^{(2)}_{0}, \tau^{(2)}_{1})\\
\overline{\beta}_{2}(\tau) \quad \text{ for } \tau \in (\tau^{(2)}_{2},1)
\end{cases}
\end{equation*}

as $n \to +\infty$, and this completes the proof of Theorem \ref{teoB}.

\qed

Next, we proceed to show that actually the given condition \eqref{618} ( or equivalently \eqref{619}) on the pair $(\beta_1, \beta_2)$ is also \underline{necessary} for the solvability of \eqref{31}-\eqref{32bis}.
According to \eqref{613} and  \eqref{613*}, it is clear that we need to be concerned only with the case:  $\tau \,\in ( 0,  \tau_1^{(1)} ).$

To this purpose we establish the following:

\begin{theo}\label{teoC}
\begin{equation}\label{51}
\begin{aligned} 
&(i)\;  \text { If } \tau \in (0, \frac{1}{2}) \; \text{ then } \beta_i \leq \beta_i^{*} \quad i=1, 2, \; \text{ and } \; \beta_2 =\varphi_1^{+}(\beta_1) \geq \max \{ \beta^{**}_2,  4 \},\\
&\text { or  equivalently } \quad \beta_1 =\varphi_2^{+}(\beta_2) \geq \max \{ \beta^{**}_1,  4(N+1) \}.\\
&(ii)\;  \text { If }\quad  \tau \in ( \frac{1}{2}, \tau^{(2)}_1) \quad \text{ then } \quad \beta_2\leq \beta_2^{**}  \quad \text{ and }\, \beta_1\geq \beta_1^{*} > 4(N+1).\\
&(iii)\; \text { If } \quad \tau \in ( \frac{1}{2}, \tau^{(1)}_1) \quad \text{ then } \quad \beta_1\leq \beta_1^{**}  \quad \text{ and }\quad  \beta_2\geq \beta_2^{*} > 4.
\end{aligned}
\end{equation}
\end{theo}

\proof It is convenient to introduce the following change of variable: $r=e^{t}$, and so consider  the functions:
\begin{equation}\label{52}
z(t):=v_1(e^{t})\quad\text{and}\quad u(t):=v_2(e^{t}), \;\;t\in\mathbb{R};
\end{equation}
satisfying:
\begin{equation}\label{53}
\begin{cases}
\frac{d^2z}{dt^2}+e^{2(N+1)t+z}-\tau e^{2t+u}=0
\quad\quad \text{for}\;\;t\in\mathbb{R} \\
\frac{d^2u}{dt^2}+e^{2t+u}-\tau e^{2(N+1)t+z}=0
\quad\quad \text{for}\;\;t\in\mathbb{R}\\
\frac{d z}{dt}(-\infty)=\frac{d
u}{dt}(-\infty)=0,\quad z(-\infty)\in\mathbb{R},\quad u(-\infty)\in\mathbb{R} \\
\int_{\mathbb{R}}e^{2(N+1)t+z}dt =\beta_1\quad\text{and}\quad
\int_{\mathbb{R}}e^{2t+u}dt =\beta_2.
\end{cases}
\end{equation}

In analogy to \eqref{220} we set: 
\begin{equation}\label{newffyugihyh}
f(t):=\int_{-\infty}^{t}e^{2(N+1)s+z(s)}ds\quad\text{and}\quad
g(t):=\int_{-\infty}^{t}e^{2s+u(s)}ds,
\end{equation}
so that, $f(t)$ and $g(t)$ define positive and strictly increasing
functions, with $f(-\infty)=g(-\infty)=0$ and 
$f(+\infty)=\beta_1$, $g(+\infty)=\beta_2.$ Furthermore,  problem \eqref{53} can be written equivalently as follows: 

\begin{equation}\label{Eq:1.9newhoioihiujkfiuugjhoihoijh}
\begin{cases}
-\frac{dz}{dt}=f(t)-\tau g(t)
\quad\quad \text{for}\;\;t\in\mathbb{R} \\
-\frac{du}{dt}=g(t)-\tau f(t)\quad\quad \text{for}\;\;t\in\mathbb{R}.
\end{cases}
\end{equation}

Moreover, by setting:
\begin{multline}\label{newffyughoioyiohkkjkgihjlkhj}
\Psi_0(t):=e^{2(N+1)t+z(t)}+e^{2t+u(t)} - 2(N+1)f(t)-2g(t)+\frac{1}{2}f^2(t)+\frac{1}{2}g^2(t)-\tau
f(t)g(t)\;\;\forall t\in\mathbb{R},
\end{multline}
\begin{equation}\label{newffyughoioyiohkkjkgihjlkhjhlkjkgh}
\Psi_1(t):=e^{2(N+1)t+z(t)}-2(N+1)f(t)+\frac{1}{2}f^2(t)\quad\forall
t\in\mathbb{R},
\end{equation}
and
\begin{equation}\label{newffyughoioyiohkkjkgihjlkhjjkhjk}
\Psi_2(t):=e^{2t+u(t)}-2g(t)+\frac{1}{2}g^2(t)\quad\forall
t\in\mathbb{R},
\end{equation}
then the conditions \eqref{222} and \eqref{223} can be expressed simply as follows:
\begin{equation}\label{Eq:1.9newhoioihiujkfiuugjhoihoijhjojuoihuyg}
\Psi_0(t)=0,\quad\Psi_1(t)>0\quad\text{and}\quad
\Psi_2(t)>0\quad\forall t\in(-\infty,+\infty].
\end{equation}
Also notice that, by virtue of  \eqref{newffyughoioyiohkkjkgihjlkhj} and
\eqref{Eq:1.9newhoioihiujkfiuugjhoihoijhjojuoihuyg}, we can derive the following identities:
\begin{multline}\label{newffyughoioyiohkkjkgihjlkhjkjgj}
0=e^{2(N+1)t+z(t)}+e^{2t+u(t)}+\frac{1}{2}f(t)\Big(f(t)-4(N+1)\Big)+\frac{1}{2}g(t)\Big(g(t)-4\Big)-\tau
f(t)g(t)\\
=e^{2(N+1)t+z(t)}+\frac{1}{2}f(t)\Big(f(t)-\big(4(N+1)+8\tau\big)\Big)+\frac{1}{2}\Big(g(t)-2\tau
f(t)\Big)\Big(g(t)-4\Big)+e^{2t+u(t)}, \quad\forall t\in\mathbb{R}.
\end{multline}

At this point, we introduce the function:
\begin{multline}\label{newffyughoioyiohkkjkgihjlkhjhiijlkjiohhiuhghfgyhgcghfjjgghhkjklhhjk}
R_{0}(t)= 2\tau e^{2(N+1)t+z(t)}\Big( g(t)-4\Big)+
\Big(f(t)-4(N+1)\Big)\bigg(e^{2(N+1)t+z(t)}+\frac{1}{2}f(t)\Big(f(t)-\big(4(N+1)+8\tau\big)\Big)\bigg)
\;\forall t\in\mathbb{R},
\end{multline}
and by using \eqref{Eq:1.9newhoioihiujkfiuugjhoihoijh}, after straightforward calculations we can check that,
\begin{equation}\label{newffyughoioyiohkkjkgihjlkhjhiiljklhhoiiuiuygvufggfgfytjkghgghlhhhkhlhh}
\frac{dR_{0}}{dt}(t)= -(1-2\tau)e^{2(N+1)t+z(t)}\bigg(\frac{1}{2}(2\tau+1)g(t)\Big(g(t)-4\Big)+e^{2t+u(t)}\bigg),\\
\;\forall t\in\mathbb{R}.
\end{equation}
Similarly for the function:
\begin{equation}\label{newffyughoioyiohkkjkgihjlkhjhiijlkjiohhiuhnewjhkklhhgkujghjhkjjhhlghjlhkjhkhj}
R_{1}(t)=2\tau e^{2t+u(t)}\left(
f(t)-(4(N+1)+8\tau)\right)\\+g(t)\left(e^{2t+u(t)}+\frac{1}{2}(g(t)-4)(g(t)-
2\tau(4(N+1)+8\tau))\right) \;\forall t\in\mathbb{R},
\end{equation}
we have:
\begin{equation}\label{newffyughoioyiohkkjkgihjlkhjhiiljklhhoiiuiuynewhjklhjhjghjhjhjflhkjhlkhj}
\frac{dR_{1}}{dt}(t)= -(1-2\tau)e^{2t+u(t)}\left\{
e^{2(N+1)t+z(t)}+\frac{1}{2}(1+2\tau) f(t)\left(f(t)-
(8\tau+4(N+1))\right)\right\}.
\end{equation}
In particular  by \eqref{newffyughoioyiohkkjkgihjlkhjhiiljklhhoiiuiuygvufggfgfytjkghgghlhhhkhlhh}
and \eqref{newffyughoioyiohkkjkgihjlkhjhiiljklhhoiiuiuynewhjklhjhjghjhjhjflhkjhlkhj} we obtain:

\begin{multline}\label{newffyughoioyiohkkjkgihjlkhjhiiljklhhoiiuiuynewhjklhjhjghjhjhjfoiyiugghjjlkhjhlhhhhjjkjhhh}
\frac{dR_{1}}{dt}(t)-\frac{dR_{0}}{dt}(t)=
(1-2\tau)e^{2(N+1)t+z(t)}\bigg(\frac{1}{2}(2\tau+1)g(t)\Big(g(t)-4)\Big)+e^{2t+u(t)}\bigg)\\
-(1-2\tau)e^{2t+u(t)}\left\{
e^{2(N+1)t+z(t)}+\frac{1}{2}(1+2\tau) f(t)\left(f(t)-
(8\tau+4(N+1))\right)\right\}\\
=(1-4\tau^2)\Bigg\{e^{2(N+1)t+z(t)}\left(\frac{1}{2}g(t)(g(t)-4)+e^{2t+u(t)}\right)\\
-e^{2t+u(t)}\left( e^{2(N+1)t+z(t)}+\frac{1}{2} f(t)(f(t)-
(8\tau+4(N+1))) \right) \Bigg\}, \;\forall t\in\mathbb{R}.
\end{multline}

We start to analyse the case where: $0< \tau<1/2,$ and by contradiction we assume that,
$$\beta_1\geq  \beta_1^{*} = 4N+1) + 8\tau.$$
Then, for the function:
\begin{equation}\label{kyhhlkjljlojgfgjhfjkhhj}
H(t):=e^{2(N+1)t+z(t)}+\frac{1}{2} f(t)\Big(f(t)-
\big(4(N+1)+ 8\tau \big)\Big) = e^{2(N+1)t+z(t)}+\frac{1}{2} f(t)(f(t)-\beta^{*}_1)
\end{equation}
we find that, 
\begin{equation}\label{kyhhlkjljlojgfgjhfjkhhjgngn}
H(+\infty) > 0 \ \text { and } \quad H'(t) = \tau e^{2(N+1)t+z(t)}(g(t)- 4).
\end{equation}
Therefore, if we let $t_1\in\mathbb{R}$ the unique value such that: $g(t_1)=4,$ we see that $H(t)$ is decreasing for  $t< t_1,$ and in particular  $H(t)<0,$
for every $t\leq t_1$.\\ 
While $H(t)$ is increasing for $t> t_1,$ and since $H(\infty) > 0$,  we find a unique $t_0 > t_1$ where the function $H$ vanishes. More precisely,
\begin{equation}\label{kyhhlkjljlojgfgjhfityuhk}
H(t_0)=0 \quad \text{ and in particular: } \quad f(t_0) < \beta^{*}_1=4(N+1) + 8\tau,   \quad  g(t_0)>4,
\end{equation}
and moreover,
\begin{equation}\label{kyhhlkjljlojgfgjhfjkhhjhlhkjkgjk}
H(t):=e^{2(N+1)t+z(t)}+\frac{1}{2} f(t)\Big(f(t)-
\big(4(N+1)+8\tau\big)\Big) < 0, \quad\forall t < t_0.
\end{equation}

As a consequence of \eqref{newffyughoioyiohkkjkgihjlkhjhiiljklhhoiiuiuynewhjklhjhjghjhjhjfoiyiugghjjlkhjhlhhhhjjkjhhh}, \eqref{Eq:1.9newhoioihiujkfiuugjhoihoijhjojuoihuyg} 
and \eqref{kyhhlkjljlojgfgjhfjkhhjhlhkjkgjk} we derive:
\begin{multline}\label{kyhhlkjljlojgfgjhfjkgjgkhjhjgfjh}
R_1(t_0)-R_0(t_0)=
(1-4\tau^2)\int\limits_{-\infty}^{t_0}\Bigg\{e^{2(N+1)t+z(t)}\bigg(\frac{1}{2}g(t)\Big(g(t)-4)\Big)+e^{2t+u(t)}\bigg)\\
-e^{2t+u(t)} \left( e^{2(N+1)t+z(t)}+\frac{1}{2} f(t)(f(t)-
(8\tau+4(N+1))) \right)\Bigg\}dt=\\
(1-4\tau^2)\int\limits_{-\infty}^{t_0}\bigg\{e^{2(N+1)t+z(t)}\Psi_2(t)
-e^{2t+u(t)}H(t)\bigg\}dt>0.
\end{multline}

On the other hand, by \eqref{newffyughoioyiohkkjkgihjlkhjhiijlkjiohhiuhghfgyhgcghfjjgghhkjklhhjk}, and \eqref{kyhhlkjljlojgfgjhfityuhk}, we see that,
\begin{equation}\label{newffyughoioyiohkkjkgihjlkhjhiijlkjiohhiuhghfgyhgcghfjjgghhkjklhhjkghjhlkk}
R_{0}(t_0):= 2\tau e^{2(N+1)t_0+z(t_0)}\Big(
g(t_0)-4\Big) > 0,
\end{equation}
while by \eqref{newffyughoioyiohkkjkgihjlkhjhiijlkjiohhiuhnewjhkklhhgkujghjhkjjhhlghjlhkjhkhj},
\eqref{newffyughoioyiohkkjkgihjlkhjkjgj} and
\eqref{kyhhlkjljlojgfgjhfityuhk} we drive:

\begin{multline}\label{newffyughoioyiohkkjkgihjlkhjhiijlkjiohhiuhnewjhkklhhgkujghjhkjjhhlghjlhkjhkhjhmghghj}
R_{1}(t_0)\leq
g(t_0)\left(e^{2t_0+u(t_0)}+\frac{1}{2}(g(t_0)-4)(g(t_0)-
2\tau(4(N+1)+8\tau))\right)\\ \leq
g(t_0)\left(e^{2t_0+u(t_0)}+\frac{1}{2}(g(t_0)-4)(g(t_0)-
2\tau
f(t_0))\right)=\\-g(t_0)\left(e^{2(N+1)t_0+z(t_0)}+\frac{1}{2}f(t_0)(f(t_0)-(4(N+1)+8\tau))\right)=0
\end{multline}
and so, 
$$R_{1}(t_0)-R_{0}(t_0)\leq 0,$$
in contradiction with \eqref{kyhhlkjljlojgfgjhfjkgjgkhjhjgfjh}. Similarly, for $0< \tau<1/2$ we can check in an analogous way that $\beta_2 < \beta^{*}_2$ and $(i)$ is established.\\

Next, we check $(ii)$, so we assume that: $\tau \in (\frac{1}{2}, \tau^{(2)}_1)$ and we proceed to show that: $\beta_2 < \beta^{**}_2.$ 
To this purpose, as above, we argue by contradiction and suppose that,
\begin{equation}\label{5*}
\beta_2 \geq \beta^{**}_2. 
\end{equation}
Therefore from, \eqref{corocoro112} we have:
\begin{equation}\label{5**}
\beta_{1}=\varphi^{+}_{2}(\beta_{2}) \leq \varphi^{+}_{2}(\beta^{**}_{2}) = \beta^{*}_{1},
\end{equation}
and so in this case, $H(+\infty) \leq 0.$ Thus, for the function $H$ in \eqref{kyhhlkjljlojgfgjhfjkhhj}  now we have that,
$$H(t) < 0 \quad\forall t\in\mathbb{R}.$$ 
At this point, by recalling that $\tau > \frac{1}{2}$, from \eqref{newffyughoioyiohkkjkgihjlkhjhiiljklhhoiiuiuynewhjklhjhjghjhjhjfoiyiugghjjlkhjhlhhhhjjkjhhh} we obtain:
\begin{equation}\label{5bis}
\frac{dR_{1}}{dt}(t)-\frac{dR_{0}}{dt}(t)=(1-4\tau^2)\bigg\{e^{2(N+1)t+z(t)}\Psi_2(t)
-e^{2t+u(t)}H(t)\bigg\} < 0.
\end{equation}
On the other hand, under the given assumption, we see that,
\begin{equation}\label{5bis*}
R_1(+\infty) - R_0(+\infty)=\frac{1}{2}\beta_2(\beta_2 - 4)(\beta_2 - \beta^{**}_{2}) - \frac{1}{2}\beta_1(\beta_1 - 4(N+1))(\beta_1 - \beta^{*}_{1}) \geq 0,  
\end {equation}
and, $R_1(-\infty) = R_0(-\infty)= 0.$ Therefore in view of \eqref{5bis} we reach again a contradiction, and $(ii)$ is established\\ 
Finally, we can check easily that  $(iii)$ follows by similar arguments, and  the proof is completed.

\qed

At this point,  Theorem\ref{teoA} readily follows as a consequence of Theorem \ref{teoB}, which takes care about the "sufficient" part, while Theorem \ref{teoC}  provides the proof of the "necessary" condition.


\end{document}